\documentclass{amsart}

\usepackage{amssymb,amsfonts,amsmath,amsthm,algorithm,tikz}

\newtheorem{theorem}{Theorem}[section]
\newtheorem{lemma}[theorem]{Lemma}

\newtheorem{corollary}[theorem]{Corollary}

\newtheorem{assumption}{Assumption}

\theoremstyle{definition}
\newtheorem{definition}[theorem]{Definition}
\newtheorem{example}[theorem]{Example}

\theoremstyle{remark}
\newtheorem{remark}[theorem]{Remark}

\numberwithin{equation}{section}

\def\sT {{\sf T}}

\newcommand{\dom}{\text{dom }}
\newcommand{\bbE}{\mathbb{E}}
\newcommand{\bbP}{\mathbb{P}}
\newcommand{\bbR}{\mathbb{R}}

\newcommand{\sint}{\text{int}\ }
\newcommand{\sbd}{\text{bd}\ }
\newcommand{\scl}{\text{cl}\ }

\newcommand{\norm}[1]{\|#1\|}
\newcommand{\ip}[1]{\langle #1 \rangle}
\newcommand\tab[1][1cm]{\hspace*{#1}}
\newcommand\stab[1][0.7cm]{\hspace*{#1}}

\begin{document}

\title{On Feasibility of Sample Average Approximation Solutions}

\author{Rui Peng Liu}
\address{School of Industrial and Systems Engineering, Georgia Institute of Technology, Atlanta, Georgia 30332}
\email{rayliu@gatech.edu}


\date{Jan 6, 2020}



\begin{abstract}
When there are infinitely many scenarios, the current studies of two-stage stochastic programming problems rely on the relatively complete recourse assumption. However, such an assumption can be unrealistic for many real-world problems. This motivates us to study general stochastic programming problems where the sample average approximation (SAA) solutions are not necessarily feasible. When the problems are convex and the true solutions lie in the interior of feasible solutions, we show that the portion of infeasible SAA solutions decays exponentially as the sample size increases. We also study functions with chain-constrained domain and show that the portion of SAA solutions having a low degree of feasibility decays exponentially as the sample size increases. This result is then extended to multistage stochastic programming.
\end{abstract}

\maketitle

\section{Introduction}\label{sec:intr}

We consider the stochastic programming problem
\begin{equation} \label{gsp}
\inf_{x\in \mathcal{X}} F(x) := \bbE[f_\xi(x)],
\end{equation}
where $\mathcal{X}\subseteq \bbR^n$ is a non-empty set; the random vector $\xi : \Omega \rightarrow \bbR^{n_0}$ is defined on the probability space $(\Omega, \mathcal{F}, \bbP)$, and the support of $\xi$ is $\Xi := \xi(\Omega)\subseteq \bbR^{n_0}$ (we may also use $\xi$ to denote the outcomes in $\Xi$); and for each outcome $\xi\in \Xi$, $f_\xi : \mathbb{R}^n\rightarrow \mathbb{R}\cup \{+\infty\}$ is an extended real-valued function. Throughout the paper, we assume $F(x) < +\infty$ for some $x\in \mathcal{X}$.

An important class of \eqref{gsp} is the two-stage stochastic programming. In two-stage problems, $f_\xi(x)$ is given by the optimal value of the second stage problem, i.e.,
\[ f_\xi(x) := \inf_{y\in \mathcal{Y}(x, \xi)}\ g_\xi(y), \]
where $\mathcal{Y}$ is a multivalued function that maps $(x, \xi)$ to sets, and $g_\xi$ is a real-valued function for each $\xi\in \Xi$. By definition, $f_\xi(x) = +\infty$ if the set $\mathcal{Y}(x, \xi)$ is empty. In two-stage stochastic programming, the first stage decisions $x$ should be implemented before a realization of the random data becomes available and hence should be independent of the random data. The second stage decisions $y$ are made after observing the random data and are functions of the data. The model has found wide applications such as transportation planning \cite{BA04}, water resource management \cite{HL00}, power production \cite{WGW12}, etc.

When $\Xi$ contains infinitely many outcomes, the current studies of two-stage problems rely on the relatively complete recourse assumption; that is, for every $x\in \mathcal{X}$ and almost every $\xi\in \Xi$, the set $\mathcal{Y}(x, \xi)$ is non-empty. In terms of $f_\xi$, the assumption states that $\mathbb{P}\{f_\xi(x) = +\infty\} = 0$ for every $x\in \mathcal{X}$. However, such an assumption can be unrealistic for many real-world applications. For example, when deciding the size of a reservoir for water supply during a potential drought, it could happen that some size is too small to store enough water. Although one can forge relatively complete recourse by passing to a penalized problem~\cite{SN05}, the optimal solutions of the penalized problem may be inferior with respect to the original formulation.

In general, it is difficult to evaluate the expectation in \eqref{gsp} directly, because the underlying distribution is usually unknown. Even if the distribution is given, it is computationally hard to evaluate high-dimensional integrals to a high accuracy. The sample average approximation (SAA) method is a sampling-based method that aims to approximate the expectation using the Monte Carlo sampling technique; theoretical results (\cite{KSH02}\cite{SDR14}) and numerical experiments (\cite{LSW06}\cite{MMW99}\cite{VAK03}) indicate that the approximations could be reasonably accurate when the SAA method is applied to the two-stage stochastic programming problems with relatively complete recourse. In this paper, we study the SAA method applied to the problems \eqref{gsp} where the probability $\mathbb{P}\{f_\xi(x) = +\infty\}$ could be positive for some $x\in \mathcal{X}$. To this end we make the following assumption:
\begin{assumption}
It is possible to generate an independent and identically distributed (i.i.d.) sample $\xi^1, \xi^2, \ldots,$ of realizations of the random vector $\xi$.
\end{assumption}

The SAA method generates a (random) sample $\xi^{[N]} := (\xi^1, \ldots, \xi^N)$ of size $N$, and approximates the expectation function $F$ by the SAA function
\[ \hat{F}_N(x) := \frac{1}{N} \sum_{i = 1}^N f_{\xi^i}(x), \]
thereby approximating the true problem by the SAA problem
\begin{equation} \label{saa}
\inf_{x\in \mathcal{X}} \hat{F}_N(x).
\end{equation}
Note that the SAA problem \eqref{saa} depends on the generated sample $\xi^{[N]}$. We assume there is a mapping that assigns a SAA solution $x^*(\xi^{[N]})$ to each SAA problem, e.g., an optimization algorithm that outputs $x^* = x^*(\xi^{[N]})$ by solving \eqref{saa}. We only assume that $x^*$ is a feasible solution of \eqref{saa} (i.e., $x^*\in \mathcal{X}$ and $\hat{F}_N(x^*) < +\infty$), while $x^*$ needs not be an optimal solution; indeed, it is unrealistic to require $x^*$ to be optimal when the SAA problem is nonconvex. In section \ref{sec:tcc}, we explicitly require $x^*$ to be optimal in stochastic convex programming.

The SAA solutions are not necessarily feasible for the true problem \eqref{gsp}; it could happen that $F(x^*(\xi^{[N]})) = +\infty$. To better understand the quality of the SAA solutions, it is important to study how feasible the solution $x^*(\xi^{[N]})$ is. The following definition of degree of feasibility involves the concept of the domain of $f_\xi$, i.e., $\dom f_\xi = \{x : f_\xi(x) < +\infty\}$.
\begin{definition}[Degree of Feasibility] \label{dof}
The degree of feasibility of an SAA solution $x^*(\xi^{[N]})$ is
\[ d(x^*(\xi^{[N]})) := \bbP\{\xi\in\Xi : x^*(\xi^{[N]})\in \dom f_\xi\}. \]
One can also view $\xi$ as a random vector independent of the random sample $\xi^{[N]}$, then $d(x^*(\xi^{[N]}))$ is the conditional probability
\[ d(x^*(\xi^{[N]})) = Pr(x^*(\xi^{[N]})\in \dom f_\xi\ |\ \xi^{[N]}). \]
We adopt the former view for other similar definitions.
\end{definition}

We intend to demonstrate in this paper that, for a broad class of problems, the portion of the SAA solutions $x^*(\xi^{[N]})$ such that $d(x^*(\xi^{[N]})) < 1 - \alpha$ (for some $\alpha\in (0, 1)$) decreases exponentially in $N$; here, the portion is measured with respect to the product measure $\mathbb{P}^N$ associated with $\xi^{[N]}$. A similar problem regarding chance-constrained stochastic programming was studied in \cite{CC05}\cite{CC08} (and references therein), but the nature of their problem is somewhat different from the one considered in this paper. The asymptotic epi-convergence of the SAA function to the expectation function was studied in \cite{DW88}. Their result establishes the convergence of the optimal SAA solutions to the optimal solutions of the true problem \eqref{gsp} when the sample size $N$ tends to infinity. However, the degree of feasibility cannot be deduced from their result when $N$ is finite.

The rest of this paper is organized as follows. In section \ref{sec:ccd}, we investigate an exponential rate of convergence of degree of feasibility when $\{f_\xi\}$ has chain-constrained domain (the proofs are delayed to Appendix \ref{sec:ptc}). In section \ref{sec:tcc}, stochastic convex programming is considered. Under the uniform convergence property of the SAA method, we show that the portion of infeasible SAA solutions decays exponentially in $N$ when the true solutions lie in the interior of the feasible region. Also, by combining uniform convergence and chain-constrained domain, we significantly improve the rate presented in section \ref{sec:erc}. In section \ref{sec:msp}, the result is extended to multistage stochastic programming.

We use the following notation and terminology throughout the paper. We denote $F^*$ and $\mathcal{X}^*$ to be the optimal value and the set of optimal solutions of the true problem \eqref{gsp}, respectively. The Euclidean norm of a vector $x\in \bbR^n$ is $\norm{x}$. For a set $U\subset \mathbb{R}^n$, we denote $\sint U, \sbd U, \scl U, U^c, |U|$ to be its interior, boundary, closure, complement, and cardinality, respectively; also, we denote the distance of $x$ to $U$ by $\text{dist}_U(x) = \text{dist}(x, U) = \inf_{y\in U} \norm{x - y}$. We denote $\mathbb{R}, \mathbb{Q}, \mathbb{Q}^c, \mathbb{N}$ to be real, rational, irrational, and natural numbers, respectively. Given two sets $U$ and $V$, we denote $U\subset V$ to be $U\subseteq V$ and $U\ne V$. The effective domain of a function $f$ is $\dom f = \{x : f(x) < +\infty\}$. For a natural number $m$, we denote $[m] := \{1, \ldots, m\}$. The preimage of a set $T$ under the function $f$ is $f^{-1}T = \{x : f(x)\in T\}$.

\section{Chain-constrained domain} \label{sec:ccd}

In this section, we investigate the degree of feasibility of SAA solutions when $\{f_\xi\}$ has chain-constrained domain. Let us first define chain-constrained domain.

\begin{definition}
{\rm
A collection of sets $\{U^\omega\}_{\omega\in I}$ is a chain if for any $\omega_1, \omega_2\in I$, either $U^{\omega_1}\subseteq U^{\omega_2}$ or $U^{\omega_1}\supseteq U^{\omega_2}$.
}
\end{definition}

If the index set $I$ is finite, then the sets can be put into order:
\[ U^{\omega_1}\subseteq U^{\omega_2}\subseteq\dots \subseteq U^{\omega_n}; \]
in particular, $\cap_{\omega\in I} U^\omega = U^{\omega_1}$, i.e., $U^{\omega_1}$ is the smallest element.

\begin{definition}[Chain-constrained Domain] \label{ccd}
{\rm
A collection of functions $\{f_\xi\}_{\xi\in \Xi}$ has chain-constrained domain of order $m\in \mathbb{N}$ if there exist $m$ collections of sets $\big\{U_k^\xi\big\}_{\xi\in \Xi}, k\in [m]$, such that each collection $\big\{U_k^\xi\big\}$ is a chain, and for each $\xi\in \Xi$,
\begin{equation} \label{ccdform}
\dom f_\xi = \cap_{k = 1}^m U_k^\xi.
\end{equation}
}
\end{definition}
\begin{remark}
{\rm
Definition \ref{ccd} is motivated by the functional constraints
\begin{equation} \label{fcdform}
\dom f_\xi = \{x : c_k(x) \le \ell_k(\xi), k\in [m]\} \stab \forall \xi\in \Xi,
\end{equation}
where for each $k$, $c_k$ is a function of $x$ and $\ell_k$ is a random variable supported on $\Xi$. Consider
\[ U_k^\xi = \{x : c_k(x) \le \ell_k(\xi)\}; \]
observe that for any $\xi_1$ and $\xi_2\in \Xi$, either $\ell_k(\xi_1)\le \ell_k(\xi_2)$ or $\ell_k(\xi_1)\ge \ell_k(\xi_2)$, which implies $U_k^{\xi_1}\subseteq U_k^{\xi_2}$ or $U_k^{\xi_1}\supseteq U_k^{\xi_2}$; thus $\{U_k^\xi\}_{\xi\in \Xi}$ is the chain induced by both the function $c_k$ and the random variable $\ell_k$. One advantage of studying the form \eqref{ccdform} instead of \eqref{fcdform} is that \eqref{ccdform} helps recognize sets that are not commonly represented as functional-constrained domain, e.g., when $\{U^\xi\}$ is a chain of discrete sets.
}
\end{remark}

Chain-constrained domain covers a broad range of two-stage stochastic programming problems.

\begin{example} \label{ex:ts}
{\rm
Suppose $f_\xi$ is given by the second stage problem
\begin{align*}
f_\xi(x) := \inf_y\ \ & g_\xi(y) \\
			\text{s.t.}\ \ & W_\xi y + T_\xi x = h_\xi \\
							& y\ge 0,
\end{align*}
where the data $(h_\xi, g_\xi, T_\xi, W_\xi)$ satisfy the conditions
\begin{enumerate} 
\item the functions $g_\xi$ are finite everywhere;
\item there are only finitely many distinct $W_\xi$ and $T_\xi$, i.e., $|\{W_\xi\}| = p$ and $|\{T_\xi\}| = q$ for some $p, q\in \mathbb{N}$ (though there is no restriction on $h_\xi$).
\end{enumerate}
Note that $f_\xi$ is convex when $g_\xi$ is a convex function.

Denote $\{W_1, \ldots, W_p\}$ and $\{T_1, \ldots, T_q\}$ to be the set of distinct matrices. Observe that $f_\xi(x) < +\infty$ if and only if the set $\{y\ge 0 : W_\xi y + T_\xi x =  h_\xi\}$ is non-empty, which by Farkas' lemma, if and only if $a^{\sT}(h_\xi - T_\xi x)\ge 0$ for all $a$ such that $a^{\sT}W_\xi\ge 0$. For each $i\in [p]$, let $\{a_{ij}\}_{j\in J_i}$ denote the set of extreme rays of the polyhedral cone $\{a : a^{\sT}W_i\ge 0\}$, then
\[ \dom f_\xi = \{x : a_{ij}^{\sT} T_k x\le a_{ij}^{\sT} h_\xi,\ W_\xi = W_i, j\in J_i, T_\xi = T_k\}\stab \forall \xi\in \Xi. \]
For each $i\in [p], j\in J_i$, and $k\in [q]$, consider the chain $\{U_{ijk}^\xi\}$ such that
\[ U_{ijk}^\xi := \begin{cases} \{x : a_{ij}^{\sT} T_k x\le a_{ij}^{\sT} h_\xi\} & \text{if } W_\xi = W_i \text{ and } T_\xi = T_k \\ \bbR^n & \text{otherwise} \end{cases}, \]
then
\[ \dom f_\xi = \cap_{i\in [p], j\in J_i, k\in [q]} U_{ijk}^\xi \stab \forall \xi\in \Xi, \]
hence $\{f_\xi\}$ has chain-constrained domain of order $q \sum_{i = 1}^{p} |J_i|$. If $W_i$ is never coupled with $T_k$ in any outcome $\xi$, then $U_{ijk}^\xi = \mathbb{R}^n$ for all $\xi$, and we can safely remove $\{U_{ijk}^\xi\}$ to reduce the order.
}
\end{example}

As illustrated by Example \ref{sep}, chain-constrained domain is more general than functional-constrained domain.

\begin{example} \label{sep}
{\rm
Let $\tau$ be an exponential random variable, then $\tau$ is supported on $(0, \infty)$. Consider the chain $\{U^\tau\}_{\tau > 0}$,
\[ U^\tau := \begin{cases} \tau \overline{B} & \text{if $\tau\in \mathbb{Q}$} \\ \tau B & \text{if $\tau\in \mathbb{Q}^c$} \end{cases} \tab \forall \tau > 0, \]
where $\overline{B}$ and $B$ are the closed and open unit intervals in $\bbR$. We claim that $\{U^\tau\}$ cannot be characterized by sublevel sets of any function $c$, i.e., there does not exist $c, \ell$ such that $U^\tau = c^{-1}(-\infty, \ell(\tau)]$ for all $\tau > 0$.

For contradiction, suppose there exist $c$ and $\ell$ such that $U^\tau = c^{-1}(-\infty, \ell(\tau)]$ for all $\tau$. For $0 < \tau_1 < \tau_2$, we have $U^{\tau_1}\subset U^{\tau_2}$, which implies $\ell(\tau_1) < \ell(\tau_2)$. Let $x$ be a positive irrational number, then $x\not\in U^x$ but $x\in \cap_{\tau > x} U^\tau$, which translates to $\ell(x) < c(x)\le \inf_{\tau > x} \ell(\tau)$. Consider the gap $\gamma_x := \inf_{\tau > x} \ell(\tau) - \ell(x)$, then $\gamma_x > 0$ for each positive irrational number $x$. For any $0 < \tau_1 < \tau_2$,
\[ \ell(\tau_2) - \ell(\tau_1)\ge \sum_{x\in \mathbb{Q}^c\cap [\tau_1, \tau_2)} \gamma_x = +\infty, \]
since the sum of uncountably many positive numbers necessarily diverges to infinity, a contradiction.
}
\end{example}

\subsection{Exponential rate of convergence} \label{sec:erc}
For a specified threshold $\alpha\in (0, 1)$, we are interested in the probability
\begin{equation} \label{lowdof}
\bbP^N\{d(x^*(\xi^{[N]})) < 1 - \alpha\}.
\end{equation}
When $\{f_\xi\}$ has chain-constrained domain, \eqref{lowdof} decreases exponentially in $N$. We bound \eqref{lowdof} by bounding the degree of feasibility of the domain of SAA function $\dom \hat{F}_N = \cap_{i = 1}^N \dom f_{\xi^i}$.
\begin{definition} \label{Dof}
{\rm
The degree of feasibility of $\dom \hat{F}_N$ is
\[ D(\xi^{[N]}) := \bbP\{\xi\in \Xi : \cap_{i = 1}^N \dom f_{\xi^i}\subseteq \dom f_\xi\}. \]
}
\end{definition}
For each sample $\xi^{[N]}$, the SAA solution $x^*(\xi^{[N]})$ is assumed to lie in the domain of the SAA function, i.e., $\hat{F}_N(x^*(\xi^{[N]})) < +\infty$. If $\cap_{i = 1}^N \dom f_{\xi^i}\subseteq \dom f_\xi$ for some $\xi$, then $x^*(\xi^{[N]})\in \dom f_\xi$. Hence, $d(x^*(\xi^{[N]}))\ge D(\xi^{[N]})$ and
\[ \bbP^N\{d(x^*(\xi^{[N]})) < 1 - \alpha\}\le \bbP^N\{D(\xi^{[N]}) < 1 - \alpha\}. \]

Suppose $\{f_\xi\}$ has chain-constrained domain of order $m$. For fixed $\xi^{[N]}$ and $\xi\in \Xi$, if $\cap_{i = 1}^N U_k^{\xi^i}\subseteq U_k^\xi$ for each $k\in [m]$, then
\[ \cap_{i = 1}^N \dom f_{\xi^i} = \cap_{k = 1}^m  (\cap_{i = 1}^N U_k^{\xi^i})\subseteq \cap_{k = 1}^m U_k^\xi = \dom f_\xi. \]
It follows that
\[ \mathfrak{D}(\xi^{[N]}) := \mathbb{P}\left\{\xi\in \Xi : \cap_{i = 1}^N U_k^{\xi^i}\subseteq U_k^\xi \text{ for all } k\in [m]\right\}\le D(\xi^{[N]}) \]
and
\[ \mathbb{P}^N\{D(\xi^{[N]}) < 1 - \alpha\}\le \mathbb{P}^N\{\mathfrak{D}(\xi^{[N]}) < 1 - \alpha\}. \]
Under the technical assumptions of Theorem \ref{ptcthm}, Corollary \ref{ptccor2} gives
\[ \mathbb{P}^N\left\{\mathfrak{D}(\xi^{[N]}) < 1 - \alpha\right\} \le \sum_{k = 0}^{m - 1} \binom{N}{k} \alpha^k (1 - \alpha)^{N - k}; \]
roughly speaking, aside from the assumptions on the measurability of some sets, Theorem \ref{ptcthm} assumes that if $Y\subseteq \Xi$ has a positive measure, then for each $k\in [m]$, there exist $\omega_1, \omega_2\in Y$ such that
\[ \mathbb{P}\{\xi\in Y : U^{\omega_1}_k \subseteq U_k^\xi\subseteq U^{\omega_2}_k\} > 0; \]
this assumption resembles the inner regularity of measures on $\mathbb{R}$. Nevertheless, if the chains are induced by functional constraints (i.e., $U_k^\xi = \{x : c_k(x)\le \ell_k(\xi)\}$ for random variables $\ell_k$), then by Corollary \ref{ptccor1} and \ref{ptccor3}, the assumptions of Theorem \ref{ptcthm} are satisfied automatically.

\begin{theorem} \label{thm:erc}
Let $m\le N$ and $\alpha\in [0, 1]$. Suppose $\{f_\xi\}_{\xi\in \Xi}$ has chain-constrained domain of order $m$, then under the assumptions of Theorem \ref{ptcthm},
\begin{equation} \label{eqn:gbc}
\mathbb{P}^N\{\mathfrak{D}(\xi^{[N]}) < 1 - \alpha\}\le \sum_{k = 0}^{m - 1} \binom{N}{k} \alpha^k (1 - \alpha)^{N - k};
\end{equation}
in particular,
\begin{equation} \label{eqn:gb}
\bbP^N\{d(x^*(\xi^{[N]})) < 1 - \alpha\}\le \mathbb{P}^N\{D(\xi^{[N]}) < 1 - \alpha\}\le \sum_{k = 0}^{m - 1} \binom{N}{k} \alpha^k (1 - \alpha)^{N - k}.
\end{equation}
If the chains are induced by functional constraints, then the assumptions of Theorem \ref{ptcthm} are satisfied automatically.
\end{theorem}

\begin{remark}
{\rm
The bound $\sum_{k = 0}^{m - 1} \binom{N}{k} \alpha^k (1 - \alpha)^{N - k}$ does not necessarily depend on the dimension of the variable $x$, which makes it potentially useful in a high-dimensional setting. In section \ref{sec:ccdr}, we show that the dependence on $m$ can be mitigated when the problem is convex.
}
\end{remark}

\begin{remark}
{\rm
The sum $\sum_{k = 0}^{m - 1} \binom{N}{k} \alpha^k (1 - \alpha)^{N - k}$ is the tail probability of the binomial distribution Bin$(N, \alpha)$. For $N\alpha\ge m - 1$, the Chernoff bound~\cite{V18} gives the estimate
\[ \sum_{k = 0}^{m - 1} \binom{N}{k} \alpha^k (1 - \alpha)^{N - k}\le \exp\left\{ -\frac{(N\alpha - m + 1)^2}{2N \alpha} \right\}. \]
}
\end{remark}

\section{The convex case} \label{sec:tcc}

Throughout this section, we assume $\mathcal{X}$ is a closed convex set, the set of optimal solutions $\mathcal{X}^*$ is non-empty and compact, and $f_\xi$ is convex for all $\xi\in \Xi$ (this implies $\dom f_\xi$ is a convex set). In addition, we assume $F > -\infty$, and hence $F$ is a convex function. We also assume that each SAA solution $x^*(\xi^{[N]})$ is an optimal solution of the SAA problem \eqref{saa}.

In the convex case, the feasibility of SAA solutions depends on the local geometry around $\mathcal{X}^*$; the main idea is to combine convexity and the uniform convergence of $\hat{F}_N$ to $F$. A result regarding uniform convergence is summarized in Theorem \ref{thm:uc}, and its proof can be found in \cite[section 7.2.10]{SDR14}.

For each $x\in \dom F$, we define $M_x(t) := \bbE[e^{t(f_\xi(x) - F(x))}]$ to be the moment generating function of the random variable $f_\xi(x) - F(x)$.

\begin{theorem}[{\cite[section 7.2.10]{SDR14}}] \label{thm:uc}
{\rm
Let $X\subseteq \dom F$ be a compact set of diameter $D$. Suppose
\begin{enumerate}
\item[(C1)] For every $x\in X$ the moment generating function $M_x(t)$ is finitely valued for all $t$ in a neighborhood of zero.
\item[(C2)] There exists a (measurable) function $\kappa : \Xi \rightarrow \bbR_+$ such that $\bbE[\kappa(\xi)] = L$ and $f_\xi$ is $\kappa(\xi)$-Lipschitz on $X$ for all $\xi$.
\item[(C3)] The moment generating function $M_\kappa(t) := \bbE[e^{t\kappa(\xi)}]$ of $\kappa(\xi)$ is finitely valued for all $t$ in a neighborhood of zero.
\end{enumerate}
Then for any $\epsilon > 0$, there exist positive constants $C$ and $\beta$, independent of $N$, such that
\[ \bbP^N\left\{\sup_{x\in X} |\hat{F}_N(x) - F(x)|\ge \epsilon \text{ or } \frac{1}{N}\sum_{i = 1}^N \kappa(\xi^i) > 2L\right\}\le Ce^{-N\beta}. \]
Moreover, if assumption (C1) is replaced by
\begin{enumerate}
\item[(C4)] There exists constant $\sigma >0$ such that for any $x\in X$, the following inequality holds:
\[ M_x(t)\le e^{\sigma^2 t^2/2} \stab \forall t\in \bbR, \]
\end{enumerate}
then for some constants $\ell$ and $\rho$,
\[ \bbP^N\left\{\sup_{x\in X} |\hat{F}_N(x) - F(x)|\ge \epsilon\right\} \le \exp(-N\ell) + 2\left[\frac{4\rho DL}{\epsilon}\right]^n \exp\left\{-\frac{N\epsilon^2}{32\sigma^2}\right\}. \]
If $\kappa(\xi)\equiv L$ for all $\xi\in \Xi$, then the term $\exp(-N\ell)$ can be omitted.
}
\end{theorem}

A thorough discussion of assumptions (C1) - (C4) can be found in \cite[section 5.3]{SDR14}. Basically, (C1), (C3) and (C4) assume the existence of moment generating functions in a neighborhood of $0$ in order to invoke the large deviation theory. Assumptions (C1) and (C3) hold, for example, if the corresponding random variables are sub-exponential~\cite{V18}.

\subsection{Solutions in the interior}

When $\mathcal{X}^*$ is contained in the interior of $\dom F$, the uniform convergence alone can guarantee that the portion of infeasible SAA solutions (i.e., $\mathbb{P}^N\{d(x^*(\xi^{[N]})) < 1\}$) decays exponentially in $N$. In particular, the result, which relies on Lemma \ref{lem:ury}, applies to general convex functions $\{f_\xi\}$ without chain-constrained domain.

\begin{lemma} \label{lem:ury}
Let $U$ be a compact convex set and $V$ be an open set such that $U\subset V$, then there exists a compact convex set $W$ such that $U\subset W\subset V$ and $U\cap \sbd W = \varnothing$.
\end{lemma}
\begin{proof}
Since $U$ is convex, the function dist$_U$ is convex, and so it is also continuous. Since $U$ is compact and $V^c$ is closed,
\[ r := \inf_{y\in V^c} \text{dist}_U(y) > 0. \]
Consider
\[ W := \{x : \text{dist}_U(x)\le r/2\} = \text{dist}_U^{-1}[0, r/2]. \]
Since $U$ is compact and dist$_U$ is convex and continuous, $W$ is a compact convex set. Moreover,
\[ U = \text{dist}_U^{-1}\{0\}\subset W\subset \text{dist}_U^{-1}[0, r)\subseteq V \]
and
\[ U\cap \sbd W = \text{dist}_U^{-1}(\{0\}\cap \{r/2\}) = \varnothing. \]
\end{proof}

\begin{theorem} \label{thm:cib}
Suppose $\{f_\xi\}$ is a collection of convex functions that satisfy assumptions (C1), (C2), and (C3) in Theorem \ref{thm:uc} on any compact set $X\subseteq \dom F$. If $\mathcal{X}^*$ is contained in the interior of $\dom F$, then there exist positive constants $C$ and $\beta$, independent of $N$, such that
\[ \bbP^N\{d(x^*(\xi^{[N]})) < 1\}\le Ce^{-N\beta}. \]
\end{theorem}
\begin{proof}
Consider
\[ \mathcal{B} := \{B : B \text{ compact convex}, \mathcal{X}^*\subseteq B\cap \mathcal{X} \subset \sint \dom F, \mathcal{X}^*\cap \sbd B = \varnothing\}. \]
By Lemma \ref{lem:ury}, $\mathcal{B}$ is non-empty. We fix a $B\in \mathcal{B}$ and let $B_\mathcal{X} := B\cap \mathcal{X}$, then $B_\mathcal{X}$ is compact and convex, and $\sbd B_\mathcal{X}\subseteq \sbd B\cup \sbd \mathcal{X}$. By the extreme value theorem, $F$ attains the minimum $F_B^*$ on the compact set $\sbd B_\mathcal{X}\cap \sbd B$.

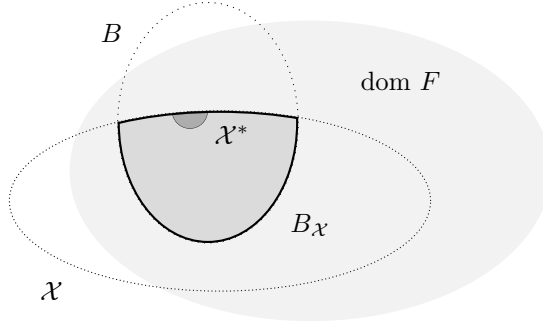
\begin{figure}[H]
\centering
\begin{tikzpicture}[scale=0.8]
		\draw [densely dotted] (-0.5, -0.5) ellipse (3.5 and 1.5);
		\draw [dotted] (-0.7, 0.8) ellipse (1.5 and 2);
		\fill [gray, fill opacity = 0.1] (1, 0) ellipse (4 and 2.5);
		
		\begin{scope}
		\clip (-1, 1) ellipse (0.3 and 0.3);
		\clip (-0.5, -0.5) ellipse (3.5 and 1.5);
		\draw (-1, 1) ellipse (0.3 and 0.3);
		\draw (-0.5, -0.5) ellipse (3.5 and 1.5);
		\fill [gray, fill opacity = 0.5] (-1, 1) ellipse (0.3 and 0.3);
		\end{scope}
		
		\begin{scope}
		\clip (-0.5, -0.5) ellipse (3.5 and 1.5);
		\clip (-0.7, 0.8) ellipse (1.5 and 2);
		\fill [gray, fill opacity = 0.2] (-0.7, 0.8) ellipse (1.5 and 2);
		\draw [ultra thick] (-0.5, -0.5) ellipse (3.5 and 1.5);
		\draw [ultra thick] (-0.7, 0.8) ellipse (1.5 and 2);
		\end{scope}
		
		\node at (-0.3, 0.6) {$\mathcal{X}^*$};
		\node at (2.5, 1.5) {$\dom F$};
		\node at (-3.3, -2) {$\mathcal{X}$};
		\node at (-2.3, 2.3) {$B$};
		\node at (1, -0.9) {$B_\mathcal{X}$};
\end{tikzpicture}
\caption{Illustration of Theorem \ref{thm:cib}.}
\end{figure}

Fix a $z\in \mathcal{X}^*$, then $F_B^* > F^* = F(z)$ since $\mathcal{X}^*$ and $\sbd B$ are disjoint; moreover, the set $X_B := \{z\}\cup (\sbd  B_\mathcal{X}\cap \sbd B)$ is compact. By Theorem \ref{thm:uc}, there exist positive constants $C_B$ and $\beta_B$ such that
\[ \bbP^N\left\{\sup_{x\in X_B} |\hat{F}_N(x) - F(x)| \ge \frac{F_B^* - F^*}{2}\right\}\le C_B e^{-N \beta_B}. \]
It remains to show
\[ \bbP^N\left\{\sup_{x\in X_B} |\hat{F}_N(x) - F(x)| \ge \frac{F_B^* - F^*}{2}\right\}\ge \bbP^N\{d(x^*(\xi^{[N]})) < 1\}. \]
Note that the uniform convergence $\sup_{x\in X_B} |\hat{F}_N(x) - F(x)| < \tfrac{F_B^* - F^*}{2}$ implies
\[ \min_{x\in \text{bd} B_\mathcal{X}\cap \text{bd} B} \hat{F}_N(x) > \min_{x\in \text{bd} B_\mathcal{X}\cap \text{bd} B} F(x) - \frac{F_B^* - F^*}{2} = F(z) + \frac{F_B^* - F^*}{2} > \hat{F}_N(z). \]
Let $x\in B_\mathcal{X}^c$, there exists the smallest $\lambda\in (0, 1]$ such that $\tilde{x} := \lambda z + (1 - \lambda) x\in \sbd B_\mathcal{X}$. If $\tilde{x}\in \sbd B$, then $\hat{F}_N(\tilde{x}) > \hat{F}_N(z)$; by the convexity of $\hat{F}_N$,
\[ \hat{F}_N(\tilde{x})\le \lambda \hat{F}_N(z) + (1 - \lambda) \hat{F}_N(x) \implies \hat{F}_N(x) > \hat{F}_N(z), \]
thus $x$ cannot be the (optimal) SAA solution. If $\tilde{x}\in (\sbd B_\mathcal{X}\setminus \sbd B)\subseteq (\sbd \mathcal{X}\setminus \sbd B)$, then $x\in \mathcal{X}^c$, i.e., $x$ is infeasible. Hence, under uniform convergence, the SAA solutions $x^*(\xi^{[N]})$ are contained in $B_\mathcal{X} \subseteq \dom F$, and
\[ \bbP^N\left\{\sup_{x\in X_B} |\hat{F}_N(x) - F(x)| < \frac{F_B^* - F^*}{2}\right\}\le \bbP^N\{d(x^*(\xi^{[N]})) = 1\}, \]
or equivalently,
\[ \bbP^N\{d(x^*(\xi^{[N]})) < 1\}\le \bbP^N\left\{\sup_{x\in X_B} |\hat{F}_N(x) - F(x)| \ge \frac{F_B^* - F^*}{2}\right\}\le C_B e^{-N \beta_B}. \]
\end{proof}

\begin{remark}
{\rm
We can modify the proof of Theorem \ref{thm:cib} to obtain different properties of the SAA solutions. Recall
\[ \mathcal{B} := \{B : B \text{ compact convex}, \mathcal{X}^*\subseteq B\cap \mathcal{X} \subset \sint \dom F, \mathcal{X}^*\cap \sbd B = \varnothing\}. \]
Consider $B\in \mathcal{B}$ such that
\begin{enumerate}
\item $B = \{x : F(x)\le F^* + \epsilon\}$ is the $\epsilon$-optimal set of $F$, then the result bounds the probability that SAA solutions are not the $\epsilon$-optimal solutions.

\item $B = \{x : \text{dist}(x, \mathcal{X}^*)\le r\}$ for some radius $r > 0$, then the result bounds the probability that SAA solutions are more than $r$ units away from $\mathcal{X}^*$.
\end{enumerate}
}
\end{remark}

\begin{corollary}
Let $B$ be an $\epsilon$-optimal set of $F$ such that $B \cap \mathcal{X}$ is compact and is contained in the interior of $\dom F$, and let $D$ denote the diameter of $B\cap \mathcal{X}$. Suppose $\{f_\xi\}$ satisfies assumptions (C2), (C3) and (C4) in Theorem \ref{thm:uc} on $B\cap \mathcal{X}$. Then
\[ \bbP^N\left\{F(x^*(\xi^{[N]})) > F^* + \epsilon\right\}\le \exp(-N\ell) + 2\left[\frac{8\rho DL}{\epsilon}\right]^n \exp\left\{-\frac{N\epsilon^2}{128\sigma^2}\right\}, \]
for constants $\ell$ and $\rho$ given in Theorem \ref{thm:uc}.
\end{corollary}

\subsection{Chain-constrained domain revisit} \label{sec:ccdr}

In this section, we investigate the case that $\mathcal{X}^*$ has a non-empty intersection with the boundary of $\dom F$. We assume $\{f_\xi\}$ has chain-constrained domain of order $m$ and
\begin{equation} \label{ccdfcd}
\dom f_\xi = \{x : c_k(x) \le \ell_k(\xi), k\in [m]\}\stab \forall \xi\in \Xi,
\end{equation}
where for each $k$, $c_k$ is a real-valued convex function and $\ell_k$ is a random variable supported on $\Xi$. For $\beta\in [0, 1)$, we define
\[ (\ell_k)_\beta := \inf\{t\in \mathbb{R} : \bbP\{\ell_k \le t\} > \beta\}. \]
By the right continuity of cumulative distribution functions, $\bbP\{\ell_k \le (\ell_k)_\beta\} \ge \beta$. Note that $(\ell_k)_0$ is the essential infimum of $\ell_k$, and that $(\ell_k)_0 > -\infty$ for each $k\in [m]$ since, by assumption, $F(x) < +\infty$ for some $x\in \mathcal{X}$.

In Theorem \ref{thm:erc}, the bound $\sum_{k = 0}^{m - 1} \binom{N}{k} \alpha^k (1 - \alpha)^{N - k}$ was derived for general functions; in stochastic convex programming, the dependence on $m$ can be mitigated. We show two different approaches, one in Example \ref{ex:ob} and the other one in Theorem \ref{cbb}.

\begin{example} \label{ex:ob}
{\rm
Let us revisit Example \ref{ex:ts}. Suppose $\{f_\xi\}$ is a collection of convex functions with the chain-constrained domain
\[ \dom f_\xi = \cap_{i\in [p], j\in J_i, k\in [q]} U_{ijk}^\xi \stab \forall \xi\in \Xi; \]
for each $i\in [p], j\in J_i, k\in [q]$,
\[ U_{ijk}^\xi := \begin{cases} \{x : a_{ijk}^{\sT} x\le b_{ij}(\xi)\} & W_\xi = W_i, T_\xi = T_k \\ \bbR^n & \text{otherwise} \end{cases}, \]
where $a_{ijk}$ are deterministic vectors and $b_{ij}$ are random variables.

Assume $\dom F = \{x : \bbP\{f_\xi(x) < +\infty\} = 1\}$, then
\[ \dom F = \{x : a_{ijk}^{\sT} x\le (b_{ij})_0, i\in [p], j\in J_i, k\in [q]\}. \]
For $\epsilon > 0$, we denote $\mathcal{X}_\epsilon := F^{-1}\{F^* + \epsilon\}$. In Figure \ref{fig1}, $\dom F$ is the polygon; for some $\epsilon_1 > \epsilon_2 > 0$, the sets $\mathcal{X}_{\epsilon_1}$ and $\mathcal{X}_{\epsilon_2}$ are the two dash-dotted curves, respectively; the solid line is part of the boundary of $\dom F$, and $\mathcal{X}^*$ is the dot on the solid line.

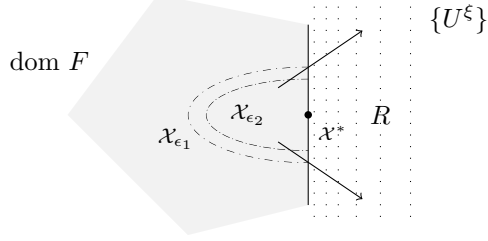
\begin{figure}[H]
\centering
\begin{tikzpicture}[scale=0.8]
      \fill [gray, fill opacity = 0.1] (2, -1.5) -- (0, -2) -- (-1, -1) -- (-2, 0) -- (-0.5, 2) -- (2, 1.5);
      \draw (2, 1.5) -- (2, -1.5);
      \node at (2, 0) [circle,fill,inner sep=1pt]{};
      \draw[loosely dotted] (2.1, 1.8) -- (2.1, -1.8);
      \draw[loosely dotted] (2.3, 1.8) -- (2.3, -1.8);
      \draw[loosely dotted] (2.5, 1.8) -- (2.5, -1.8);
      \draw[loosely dotted] (2.8, 1.8) -- (2.8, -1.8);
      \draw[loosely dotted] (3.2, 1.8) -- (3.2, -1.8);
      \draw[loosely dotted] (3.7, 1.8) -- (3.7, -1.8);
      \draw[->] (1.5, 0.45) -- (2.9, 1.4);
      \draw[->] (1.5, -0.45) -- (2.9, -1.4);
      
      \begin{scope}
		\clip (2, -1.5) -- (0, -2) -- (-1, -1) -- (-2, 0) -- (-0.5, 2) -- (2, 1.5);
		\clip (2, 0) ellipse (2 and 0.8);
		\draw [dashdotted] (2, 0) ellipse (2 and 0.8);
		\end{scope}
		
      \begin{scope}
		\clip (2, -1.5) -- (0, -2) -- (-1, -1) -- (-2, 0) -- (-0.5, 2) -- (2, 1.5);
		\clip (2, 0) ellipse (1.7 and 0.6);
		\draw [densely dashdotted] (2, 0) ellipse (1.7 and 0.6);
		\end{scope}
      
      \node at (-2.3, 0.9) {$\dom F$};
      \node at (-0.2, -0.4) {\footnotesize $\mathcal{X}_{\epsilon_1}$};
      \node at (1, 0) {\footnotesize $\mathcal{X}_{\epsilon_2}$};
      \node at (2.4, -0.3) {\scriptsize $\mathcal{X}^*$};
      \node at (3.2, 0) {$R$};
      \node at (4.5, 1.6) {$\{U^\xi\}$};
    \end{tikzpicture}
\caption{Illustration of Example \ref{ex:ob}. The arrows pass through $\mathcal{X}_{\epsilon_2}$ and $\mathcal{X}_{\epsilon_1}$ consecutively. The dotted lines correspond to the half spaces in the chain $\{U^\xi\}$.} \label{fig1}
\end{figure}

Let $\xi^{[N]}$ be a sample such that $\hat{F}_N$ approximates $F$ uniformly on $\mathcal{X}_{\epsilon_1}\cup \mathcal{X}_{\epsilon_2}$ by an error less than $(\epsilon_1 - \epsilon_2)/2$, i.e,
\[ \sup_{x\in \mathcal{X}_{\epsilon_1}\cup \mathcal{X}_{\epsilon_2}} |F(x) - \hat{F}_N(x)| < \frac{\epsilon_1 - \epsilon_2}{2}, \]
then
\[ \inf_{x\in \mathcal{X}_{\epsilon_1}} \hat{F}_N(x) > \inf_{x\in \mathcal{X}_{\epsilon_1}} F(x) - \frac{\epsilon_1 - \epsilon_2}{2} = \sup_{x\in \mathcal{X}_{\epsilon_2}} F(x) + \frac{\epsilon_1 - \epsilon_2}{2} > \sup_{x\in \mathcal{X}_{\epsilon_2}} \hat{F}_N(x). \]
Let $x_1\in \mathcal{X}_{\epsilon_1}$, $x_2\in \mathcal{X}_{\epsilon_2}$, and $\lambda\in (0, 1)$, and consider $x$ such that $x_1 = \lambda x_2 + (1 - \lambda) x$. Since $\hat{F}_N(x_2) < \hat{F}_N(x_1)$,
\[ \hat{F}_N(x_1)\le \lambda \hat{F}_N(x_2) + (1 - \lambda) \hat{F}_N(x) \implies \hat{F}_N(x) > \hat{F}_N(x_2), \]
which implies $x$ cannot be the (optimal) SAA solution; in other words, if one draws an arrow that passes through $\mathcal{X}_{\epsilon_2}$ and $\mathcal{X}_{\epsilon_1}$ consecutively, then it cannot point toward an SAA solution. In Figure \ref{fig1}, the possible location of SAA solutions is in either $\dom F$ or a region $R$ between the arrows. Suppose the dotted lines correspond to the chain $U^\xi = \{x : a^{\sT} x\le b(\xi)\}$ and the solid line is the hyperplane $\{x : a^{\sT} x = (b)_0\}$. Let $\alpha > 0$ such that
\[ R\cap \{x : a^{\sT} x\le (b)_\alpha\} \subseteq \{x : a_{ijk}^{\sT} x\le (b_{ij})_0, a_{ijk}\ne a\}, \]
i.e., the set $R\cap \{x : a^{\sT} x\le (b)_\alpha\}$ does not intersect the hyperplanes in other chains. If $\xi^{[N]}$ is a sample such that $\min_{i\in [N]} b(\xi^i)\le (b)_\alpha$, then the SAA solution $x^*(\xi^{[N]})$ lies in either $R\cap \{x : a^{\sT} x\le (b)_\alpha\}$ or $\dom F$, which implies $d(x^*(\xi^{[N]})) \ge 1 - \alpha$. Note that a random sample $\xi^{[N]}$ satisfies $\min_{i\in [N]} b(\xi^i)\le (b)_\alpha$ with probability at least $1 - (1 - \alpha)^N\ge 1 - e^{-\alpha N}$; if the aforementioned uniform approximation occurs with probability $1 - O(e^{-cN})$ for some constant $c$, then by a union bound,
\[ \mathbb{P}^N\{d(x^*(\xi^{[N]})) < 1 - \alpha\} = O(e^{-\min\{\alpha, c\} N}). \]
}
\end{example}

The following theorem adopts an approach similar to that of Theorem \ref{thm:cib}. It states that, under regularity conditions, the $m$ in the bound $\sum_{k = 0}^{m - 1} \binom{N}{k} \alpha^k (1 - \alpha)^{N - k}$ can be replaced by $|J|$, where $J$ is the index set of active constraints at $\mathcal{X}^*$, i.e.,
\[ J := \{k\in [m] : \exists x\in \mathcal{X}^*, c_k(x) = (\ell_k)_0\}. \]

\begin{theorem} \label{cbb}
Let $\dom f_\xi$ be in the form \eqref{ccdfcd} and let $J$ denote the index set of active constraints at $\mathcal{X}^*$. Suppose there exists a function $\kappa : \Xi \rightarrow \bbR_+$ such that $\bbE[\kappa(\xi)] = L$ and $f_\xi$ is $\kappa(\xi)$-Lipschitz continuous on $\dom f_\xi$ for every $\xi\in \Xi$. Suppose in addition that $\{f_\xi\}$ satisfies assumptions (C1) and (C3) in Theorem \ref{thm:uc} on any compact set $X\subset \dom F$. Then there exist $\bar{\alpha} > 0$ and positive constants $C$ and $\beta$ (independent of $N$) such that for each $N\ge |J|$ and $\alpha\in [0, \bar{\alpha})$,
\[ \mathbb{P}\{d(x^*(\xi^{[N]})) < 1 - \alpha\}\le Ce^{-N\beta} + \sum_{k = 0}^{|J| - 1} \binom{N}{k} \alpha^k (1 - \alpha)^{N - k}. \]
\end{theorem}
\begin{proof}
Let $x_1\in\dom F$, and consider $x_2$ such that $\mathbb{P}\{f_\xi(x_2) < +\infty\} = 1$, then
\[ |F(x_1) - F(x_2)|\le \mathbb{E}|f_\xi(x_1) - f_\xi(x_2)|\le \mathbb{E}[\kappa(\xi)\norm{x_1 - x_2}] = L\norm{x_1 - x_2}; \]
this implies $F$ is $L$-Lipschitz continuous on $\dom F$, and
\begin{align*}
\dom F &= \{x : \mathbb{P}\{f_\xi(x) < +\infty\} = 1\} \\
		&= \{x : c_k(x) \le (\ell_k)_0 \text{ for all } k\in [m]\},
\end{align*}
which is a closed convex set.

Since $c_k$ is a continuous function, $c_k^{-1}(-\infty, (\ell_k)_0)$ is an open set. For any $k\not\in J$, $c_k^{-1}\{(\ell_k)_0\}\cap \mathcal{X}^* = \varnothing$, hence $\mathcal{X}^*\subseteq \cap_{k\not\in J} c_k^{-1}(-\infty, (\ell_k)_0)$. Consider
\[ \mathcal{B} := \{B :B \text{ compact convex}, \mathcal{X}^*\subseteq B\cap \mathcal{X}\subset \cap_{k\not\in J} c_k^{-1}(-\infty, (\ell_k)_0), \mathcal{X}^*\cap \sbd B = \varnothing\}. \]
By Lemma \ref{lem:ury}, $\mathcal{B}$ is non-empty. We fix a $B\in \mathcal{B}$ and let $B_\mathcal{X} := B\cap \mathcal{X}$, then $B_\mathcal{X}$ is compact and convex, and $\sbd B_\mathcal{X}\subseteq \sbd B\cup \sbd \mathcal{X}$. Fix a $z\in \mathcal{X}^*$, we proceed by considering the following three disjoint  subsets of $\sbd B_\mathcal{X}$:

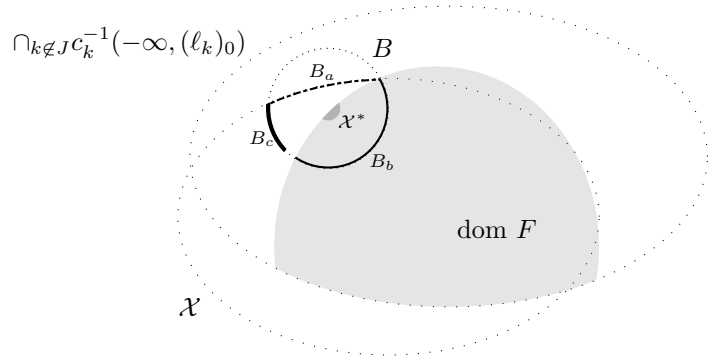
\begin{figure}[H]
\centering
\begin{tikzpicture}[scale=0.8]
		\draw [loosely dotted] (-1, 0) ellipse (3.5 and 2.3);
		\draw [loosely dotted] (0, 1) ellipse (4.3 and 2.5);
		\draw [dotted] (-2, 1.8) ellipse (1 and 1);
		
		\begin{scope}
		\clip (0, 1) ellipse (4.3 and 2.5);
		\fill [gray, fill opacity = 0.2] (-0.2, -0.5) ellipse(2.7 and 3);
		\end{scope}
		
		\begin{scope}
		\clip (0, 1) ellipse (4.3 and 2.5);
		\clip (-0.2, -0.5) ellipse(2.7 and 3);
		\clip (-2, 1.8) ellipse (1 and 1);
		\clip (-1, 0) ellipse (3.5 and 2.3);
		\draw [ultra thick] (-2, 1.8) ellipse (1 and 1);
		\end{scope}
		
		\begin{scope}
		\clip (-2, 1.8) ellipse (1 and 1);
		\clip (-1, 0) ellipse (3.5 and 2.3);
		\draw [ultra thick, densely dashdotted] (-1, 0) ellipse (3.5 and 2.3);
		\end{scope}
		
		\begin{scope}
		\clip (-3.5, 2.8) -- (-3.5, 0) -- (-1.5, 2.8) -- cycle;
		\clip (-1, 0) ellipse (3.5 and 2.3);
		\draw [ultra thick] (-2, 1.8) ellipse (1 and 1);
		\end{scope}
		
		\begin{scope}
		\clip (0, 1) ellipse (4.3 and 2.5);
		\clip (-0.2, -0.5) ellipse(2.7 and 3);
		\fill [gray, fill opacity = 0.5] (-2, 1.8) ellipse (0.2 and 0.2);
		\end{scope}
		
		\node at (-4.3, -1.5) {$\mathcal{X}$};
		\node at (-1.1, 2.8) {$B$};
		\node at (-2.1, 2.4) {\scriptsize $B_a$};
		\node at (-1.1, 0.9) {\scriptsize $B_b$};
		\node at (-3.1, 1.3) {\scriptsize $B_c$};
		\node at (-1.6, 1.6) {\scriptsize $\mathcal{X}^*$};
		\node at (0.8, -0.2) {$\dom F$};
		\node at (-5.3, 2.9) {$\cap_{k\not\in J} c_k^{-1}(-\infty, (\ell_k)_0)$};
\end{tikzpicture}
\caption{Illustration of Theorem \ref{cbb}.}
\end{figure}

\begin{enumerate}

\item $B_a := \sbd B_\mathcal{X} \setminus \sbd B$. Note that $B_a\subseteq \sbd \mathcal{X}\setminus \sbd B$.

\item $B_b := \sbd B_\mathcal{X} \cap \sbd B \cap \dom F$.

Let
\[ \epsilon := \min_{x\in B_b} F(x) - F^* > 0 \]
Consider the compact set $\{z\}\cup B_b$; by Theorem \ref{thm:uc}, for some positive constants $C_B$ and $\beta_B$,
\begin{equation} \label{eqq0}
\bbP^N\left\{\sup_{x\in \{z\}\cup B_b} |\hat{F}_N(x) - F(x)|\ge \frac{\epsilon}{3} \text{ or } \frac{1}{N}\sum_{i = 1}^N \kappa(\xi^i) > 2L\right\}\le C_B e^{-N\beta_B}.
\end{equation}
Since the complement of the event in \eqref{eqq0} implies
\begin{equation} \label{ceq1}
\min_{x : \text{dist}(x, B_b)\le \epsilon/6L} \hat{F}_N(x) > \hat{F}_N(z),
\end{equation}
we have
\begin{equation} \label{eq1}
\bbP^N\left\{\min_{x : \text{dist}(x, B_b)\le \epsilon/6L} \hat{F}_N(x) \le \hat{F}_N(z)\right\}\le C_B e^{-N\beta_B}.
\end{equation}

\item $B_c := \sbd B_\mathcal{X}\cap \sbd B\cap \{x : \text{dist}(x, B_b)\ge \epsilon/6L\}$.

We define
\begin{equation} \label{gamma}
\gamma(x) := \mathbb{P}\{\ell_j < c_j(x) \text{ for some } j\in J\}.
\end{equation}
Since $B_c\cap B_b = \varnothing$, we have $B_c\cap \dom F = \varnothing$ and $\gamma > 0$ on $B_c$. If we can show that $\gamma$ is lower semi-continuous, then $\gamma$ attains its minimum $\bar{\alpha} > 0$ on the compact set $B_c$. To see that $\gamma$ is lower semi-continuous: let $x\in \mathbb{R}^n$ and $\eta > 0$, for every $j\in J$, we can find a $t_j < c_j(x)$ such that $\mathbb{P}\{t_j\le \ell_j < c_j(x)\}\le \tfrac{\eta}{|J|}$; since $c_j$ is continuous, there exists $\delta_j > 0$ such that $c_j(y)\ge t_j$ whenever $\norm{y - x} < \delta_j$; let $\delta = \min_{j\in J} \delta_j$, then for all $y$ such that $\norm{y - x} < \delta$, we have
\[ \gamma(y) \ge \gamma(x) - \sum_{j\in J} \mathbb{P}\{t_j\le \ell_j < c_j(x)\} \ge \gamma(x) - \eta; \]
since $\eta > 0$ is arbitrary, $\gamma$ is lower semi-continuous at every $x\in \mathbb{R}^n$.

\end{enumerate}

We give conditions below that the (optimal) SAA solutions are contained in $B_\mathcal{X}$. Let $x\in B_\mathcal{X}^c$, there exists the smallest $\lambda\in (0, 1]$ such that $\tilde{x} := \lambda z + (1 - \lambda) x\in \sbd B_\mathcal{X}$. If $\tilde{x}\in B_a$, then $x\in \mathcal{X}^c$, i.e., $x$ is infeasible. We now assume $\tilde{x}\in \sbd B_\mathcal{X}\setminus B_a\subseteq \sbd B$. Suppose $\text{dist}(\tilde{x}, B_b)\le \epsilon/6L$ and $\min_{x : \text{dist}(x, B_b)\le \epsilon/6L} \hat{F}_N(x) > \hat{F}_N(z)$; since $\hat{F}_N(\tilde{x}) > \hat{F}_N(z)$, the convexity of $\hat{F}_N$ implies $\hat{F}_N(x) > \hat{F}_N(z)$, thus $x$ is not the SAA solution. It remains to address the case $\tilde{x}\in B_c$. We first claim that $\gamma(x)\ge \gamma(\tilde{x})$; indeed, suppose $c_j(\tilde{x}) > \ell_j(\xi)\ge (\ell_j)_0\ge c_j(z)$ for some $\xi\in \Xi$ and $j\in J$ (note that $\ell_j\ge (\ell_j)_0\ a.s.$), then by the convexity of $c_j$, we have $c_j(x) > c_j(\tilde{x}) > \ell_j(\xi)$; it follows that $\gamma(x)\ge \gamma(\tilde{x})$. Let $\alpha\in (0, \bar{\alpha})$, and consider a sample $\xi^{[N]}$ such that
\[ \mathfrak{D}_r(\xi^{[N]}) := \mathbb{P}\left\{\xi\in \Xi : \min_{i\in [N]} \ell_j(\xi^i)\le \ell_j(\xi) \text{ for all } j\in J\right\}\ge 1 - \alpha, \]
then
\[ \mathbb{P}\left\{\ell_j < \min_{i\in [N]} \ell_j(\xi^i) \text{ for some } j\in J\right\} = 1 - \mathfrak{D}_r(\xi^{[N]}) < \bar{\alpha} \le \gamma(\tilde{x})\le \gamma(x), \]
which by \eqref{gamma} implies $c_j(x) > \min_{i\in [N]} \ell_j(\xi^i)$ for some $j\in J$; thus $x\not\in \dom \hat{F}_N$ cannot be the SAA solution. To summarize, we have $x^*(\xi^{[N]})\in B_\mathcal{X}$ under the following two conditions:
\begin{itemize}
\item $\min_{x : \text{dist}(x, B_b)\le \epsilon/6L} \hat{F}_N(x) > \hat{F}_N(z)$;

\item $\mathfrak{D}_r(\xi^{[N]})\ge 1 - \alpha$.
\end{itemize}
Moreover, if $x^*(\xi^{[N]})\in B_\mathcal{X}\subset \cap_{j\not\in J} c_j^{-1}(-\infty, (\ell_j)_0)$, then
\begin{equation}\label{eqdD}
d(x^*(\xi^{[N]})) = \mathbb{P}\{\xi\in \Xi : c_j(x^*(\xi^{[N]}))\le \ell_j(\xi) \text{ for all } j\in J\} \ge \mathfrak{D}_r(\xi^{[N]}).
\end{equation}

By Corollary \ref{ptccor3},
\begin{equation}\label{eqDr}
\mathbb{P}^N\{\mathfrak{D}_r(\xi^{[N]}) < 1 - \alpha\}\le \sum_{k = 0}^{|J| - 1} \binom{N}{k} \alpha^k (1 - \alpha)^{N - k}.
\end{equation}
Combining \eqref{eq1}, \eqref{eqdD} and \eqref{eqDr},
\begin{align*}
	&\ \mathbb{P}^N\{d(x^*(\xi^{[N]})) < 1 - \alpha\} \\
\le&\ \mathbb{P}^N\left\{\min_{x : \text{dist}(x, B_b)\le \epsilon/6L} \hat{F}_N(x) \le \hat{F}_N(z) \text{ or } \mathfrak{D}_r(\xi^{[N]}) < 1 - \alpha\right\} \\
\le&\ C_B e^{-N\beta_B} + \sum_{k = 0}^{|J| - 1} \binom{N}{k} \alpha^k (1 - \alpha)^{N - k}.
\end{align*}
\end{proof}

In the remaining part of this section, we discuss the value of $|J|$. Recall the following result from convex optimization:
\begin{lemma}[{\cite[p. 234]{SDR14}}] \label{lem:sc}
Suppose $f$ and $f_j, j\in [m]$, are real-valued convex functions on $\mathbb{R}^n$, then there exists an index set $I\subseteq [m]$ of size at most $n$ such that the optimal solutions of the convex optimization problem
\[ \min_x \{f(x) : f_j(x)\le 0, j\in [m]\} \]
are also the optimal solutions of
\[ \min_x \{f(x) : f_j(x)\le 0, j\in I\}. \]
Such index set $I$ of the minimum cardinality is called the support indices.
\end{lemma}

Let us reinterpret stochastic convex programming problems so that Lemma \ref{lem:sc} applies. For each $\xi\in \Xi$, let $\hat{f}_\xi$ be a convex extension of $f_\xi$ to $\bbR^n$. For example, if $f_\xi$ is $\kappa(\xi)$-Lipschitz continuous on $\dom f_\xi$, then
\[ \hat{f}_\xi(x) := \inf_{y\in \dom f_\xi} \left\{f_\xi(y) + \kappa(\xi)\norm{y - x}\right\} \]
is a $\kappa(\xi)$-Lipschitz convex extension of $f_\xi$ by Kirszbraun Theorem~\cite[3.3.9]{NP18}. Let $I_\xi$ denote the indicator function of $\dom f_\xi$. If $\mathbb{E}\big[\hat{f}_\xi(x)\big] > -\infty$ on $\mathbb{R}^n$, then
\[ F(x) = \bbE\left[f_\xi(x)\right] = \bbE\big[\hat{f}_\xi(x) + I_\xi(x)\big] = \bbE\big[\hat{f}_\xi(x)\big] + \bbE\left[I_\xi(x)\right], \]
where
\[ \bbE\left[I_\xi(x)\right] = \begin{cases} 0 & \text{if } c_k(x) \le (\ell_k)_0 \text{ for all } k\in [m] \\ +\infty & \text{otherwise} \end{cases}. \]
Hence,
\[ \min_{x\in \mathcal{X}} F(x) \equiv \min_{x\in \mathcal{X}} \left\{\bbE\big[\hat{f}_\xi(x)\big] : c_k(x) - (\ell_k)_0 \le 0, k\in [m]\right\}, \]
and the two optimization problems share the same set of optimal solutions $\mathcal{X}^*$. Since $\mathbb{E}\big[\hat{f}_\xi(x)\big]$ is a convex function in $x$, if the active constraints at $\mathcal{X}^*$ correspond to the support indices, then $|J|\le n$ by Lemma \ref{lem:sc}.

\section{Feasibility in multistage stochastic programming} \label{sec:msp}

In this section, we extend the notion of degree of feasibility to multistage stochastic programming. Let $T\ge 2$ be the number of stages,  and let $\xi_t, t = 2, \ldots, T$, be a stagewise independent data process, i.e., $\xi_t$'s are mutually independent random vectors defined on the same probability space $(\Omega, \mathcal{F}, \bbP)$ (we may also use $\xi_t$ to denote the outcomes in $\Xi_t := \xi_t(\Omega)$). We denote $\xi_{[t]} := (\xi_1, \ldots, \xi_t)$ to be the history of the process up to time $t$. Consider the $T$-stage stochastic programming problem
\begin{equation} \label{msp}
\begin{split}
\min_{x_1, \pmb{x}_2, \ldots, \pmb{x}_T}\ & \mathbb{E}\left[f_1(x_1) + f_2(\pmb{x}_2(\xi_{[2]}), \xi_2) + \dots + f_T(\pmb{x}_T(\xi_{[T]}), \xi_T)\right] \\
\text{s.t.}\ & x_1\in \mathcal{X}_1, \pmb{x}_t(\xi_{[t]})\in \mathcal{X}_t(\pmb{x}_{t - 1}(\xi_{[t - 1]}), \xi_t), t = 2, \ldots, T.
\end{split}
\end{equation}
Here, the decision vectors $x_t = \pmb{x}_t(\xi_{[t]}) \in \mathbb{R}^{n_t}, t\in [T]$, are functions of the data process $\xi_{[t]}$ up to time $t$; $f_t$ are real-valued functions; and
\[ \mathcal{X}_t(x_{t - 1}, \xi_t) := \{ x_t \ge 0 : A_t x_t + B_t x_{t - 1} = b_{\xi_t} \}, \]
where $A_t$ and $B_t$ are deterministic matrices, and $b_{\xi_t}$ is a random vector supported on $\Xi_t$. The first stage data $f_1$ and $\mathcal{X}_1$ are deterministic. A thorough discussion of multistage stochastic programming can be found in~\cite{SDR14}.

The identical conditional sampling generates an i.i.d. sample $\xi_t^{[N_t]} = (\xi_t^1, \ldots, \xi_t^{N_t})$ of $\xi_t$ for each $t = 2, \ldots, T$. The (directed) scenario tree generated by the sample $\xi^{[N, T]} := (\xi_2^{[N_2]}, \ldots, \xi_T^{[N_T]})$ is constructed by first creating the root node $\xi_1$ and then creating an arc from each node $\xi_{t - 1}^1, \ldots, \xi_{t - 1}^{N_{t - 1}}$ to each node $\xi_t^1, \ldots, \xi_t^{N_t}$, $t = 2, \ldots, T$. The multistage stochastic programming problem induced by the original problem \eqref{msp} on the scenario tree is viewed as the SAA problem of \eqref{msp}; note that the data process of the induced SAA problem is still stagewise independent. We denote $\mathcal{P}_t(\xi^{[N, T]})$ to be the collection of paths from the root to a node at stage $t$, i.e., each $p_{t}\in \mathcal{P}_{t}(\xi^{[N, T]})$ has the form $p_{t} = (\xi_1, \xi_2^{i_2}, \ldots, \xi_{t}^{i_{t}}), i_s\in [N_s]$.

For a generated sample $\xi^{[N, T]}$, we assume there is a mapping that assigns a decision $x_{t}^*(p_{t})$ to each $p_{t} = (\xi_1, \xi_2^{i_2}, \ldots, \xi_{t}^{i_{t}}) \in \mathcal{P}_{t}(\xi^{[N, T]})\ \forall t\in [T]$, such that $x_1^* = x_1^*(\xi_1)\in \mathcal{X}_1$, and for the subpath $p_{s - 1} = (\xi_1, \xi_2^{i_2}, \ldots, \xi_{s - 1}^{i_{s - 1}})$, $s = 2, \ldots, t$, we have $x_s^*(p_s)\in \mathcal{X}_s(x_{s - 1}^*(p_{s - 1}), \xi_s^{i_s})$; for example, one such mapping is given by an optimization algorithm that solves the SAA problem and outputs feasible solutions. It is important to note that the solutions $x_t^*(p_t)$ depend on both the sample $\xi^{[N, T]}$ and the specific path $p_t$ in the scenario tree generated by the sample $\xi^{[N, T]}$.

The topology of the scenario tree (e.g., $\mathcal{P}_t(\xi^{[N, T]})$) is determined by the parameters $N_2, \ldots, N_T$. By fixing $N_t$'s and viewing $\xi^{[N, T]}$ as random samples, we treat each $x^*_t(p_t)$ as a random vector with respect to $\xi^{[N, T]}$. The relatively complete recourse for multistage stochastic programming states that for every $t\le T - 1$, it holds that for almost every $\xi_{[t]}$ and any sequence of feasible solutions $x_i = \pmb{x}_i(\xi_{[i]}), i = 1, \ldots, t$, the set $\mathcal{X}_{t + 1}(x_t, \xi_{t + 1})$ is non-empty for almost every $\xi_{t + 1}$. Without relatively complete recourse, we are interested in the probability that $\mathcal{X}_t(x_{t - 1}^*(p_{t - 1}), \xi_t)$ is non-empty.

\begin{definition} \label{mdof}
{\rm
Let $\xi^{[N, T]}$ be a generated sample. For each $t = 2, \ldots, T$ and $p\in \mathcal{P}_{t - 1}(\xi^{[N, T]})$, the degree of feasibility of $x_{t - 1}^*(p)$ is
\[ d_t(x_{t - 1}^*(p)) := \bbP\{\xi_t\in \Xi_t : \mathcal{X}_t(x_{t - 1}^*(p), \xi_t) \text{ is non-empty}\}. \]
}
\end{definition}

Definition \ref{mdof} coincides with Definition \ref{dof} when $t = T = 2$. We can apply Theorem \ref{thm:erc} to obtain a probabilistic bound on $\min_{p\in \mathcal{P}_{t - 1}(\xi^{[N, T]})} d_t(x_{t - 1}^*(p))$, which is the minimum degree of feasibility over all solutions at stage $t$; here the probability measure is the product measure $\bbP^{N, T} = \bbP^{N_2}\times \dots \times \bbP^{N_T}$ associated with the random sample $\xi^{[N, T]}$.

\begin{corollary} \label{cor:multi}
Consider a $T$-stage problem of the form $\eqref{msp}$; we denote  $m_t$ be the number of extreme rays of the cone $\{r : r^\sT A_t \ge 0\}$ for each $t = 2, \ldots, T$. Let $\xi^{[N, T]} = (\xi_2^{[N_2]}, \ldots, \xi_T^{[N_T]})$ be a random sample such that $N_t\ge m_t$. Then for $I\subseteq \{2, \ldots, T\}$ and $\alpha_t\in [0, 1]$,
\[ \bbP^{N, T}\left\{\min_{p\in \mathcal{P}_{t - 1}(\xi^{[N, T]})} d_t(x_{t - 1}^*(p))\ge 1 - \alpha_t, t\in I \right\}\ge \prod_{t\in I} \left(\sum_{k = m_t}^{N_t} \binom{N_t}{k} \alpha_t^k (1 - \alpha_t)^{N_t - k}\right); \]
in particular, for $\alpha\in [0, 1]$ and $t = 2, \ldots, T$,
\[ \bbP^{N, T}\left\{\min_{p\in \mathcal{P}_{t - 1}(\xi^{[N, T]})} d_t(x_{t - 1}^*(p)) < 1 - \alpha\right\}\le \sum_{k = 0}^{m_t - 1} \binom{N_t}{k} \alpha^k (1 - \alpha)^{N_t - k}. \]
\end{corollary}
\begin{proof} For $t = 2, \ldots, T$, consider
\[ S_{\xi_t} := \{x_{t - 1} : \mathcal{X}_t(x_{t - 1}, \xi_t) \text{ is non-empty}\}, \]
then $d_t(x_{t - 1}^*(p)) = \mathbb{P}\{\xi_t\in \Xi_t : x_{t - 1}^*(p)\in S_{\xi_t}\}$. We define
\[ D\left(\xi_t^{[N_t]}\right) := \bbP\{\xi_t\in \Xi_t : \cap_{i = 1}^{N_t} S_{\xi_t^i} \subseteq S_{\xi_t}\}. \]
Since $x_{t - 1}^*(p)\in S_{\xi_t^i}$ for each $p\in \mathcal{P}_{t - 1}(\xi^{[N, T]})$ and $i = 1, \ldots, N_t$, we have
\[ \min_{p\in \mathcal{P}_{t - 1}(\xi^{[N, T]})} d_t(x_{t - 1}^*(p))\ge D\left(\xi_t^{[N_t]}\right). \]
By Farkas' lemma, $S_{\xi_t} = \{x : r_i^\sT B_t x\le r_i^\sT b_{\xi_t}\ \forall i\}$, where $\{r_i\}$ is the set of extreme rays of the polyhedral cone $\{r : r^\sT A_t \ge 0\}$. By viewing $\{S_{\xi_t}\}$ as a chain-constrained domain of order $m_t$, we can apply Theorem \ref{thm:erc} to obtain the bound,
\[ \bbP^{N, T}\left\{D\left(\xi_t^{[N_t]}\right) < 1 - \alpha\right\} = \bbP^{N_t}\left\{D\left(\xi_t^{[N_t]}\right) < 1 - \alpha\right\} \le \sum_{k = 0}^{m_t - 1} \binom{N_t}{k} \alpha^k (1 - \alpha)^{N_t - k}; \]
by stagewise independence,
\[ \bbP^{N, T}\left\{D\left(\xi_t^{[N_t]}\right)\ge 1 - \alpha_t, t\in I\right\}\ge \prod_{t\in I} \left(\sum_{k = m_t}^{N_t} \binom{N_t}{k} \alpha_t^k (1 - \alpha_t)^{N_t - k}\right). \]
Therefore,
\begin{align*}
	&\ \bbP^{N, T}\left\{\min_{p\in \mathcal{P}_{t - 1}(\xi^{[N, T]})} d_t(x_{t - 1}^*(p))\ge 1 - \alpha_t, t\in I \right\} \\
\ge &\ \bbP^{N, T}\left\{D\left(\xi_t^{[N_t]}\right)\ge 1 - \alpha_t, t\in I\right\} \\
\ge &\ \prod_{t\in I} \left(\sum_{k = m_t}^{N_t} \binom{N_t}{k} \alpha_t^k (1 - \alpha_t)^{N_t - k}\right).
\end{align*}
\end{proof}

\section{Conclusions} In situations where the SAA solutions could be infeasible to the true problem \eqref{gsp}, it is shown that for functions with chain-constrained domain, the portion of SAA solutions having a low degree of feasibility decays exponentially when the sample size increases. For convex problems, estimates of this rate can be improved; in particular, when the true solutions are contained in the interior of the domain of the expectation function, the portion of infeasible SAA solutions decays exponentially as the sample size increases, and this result holds even without chain-constrained domain. In the multistage case, the result is extended to bound the minimum degree of feasibility of SAA solutions over all sample paths.

\appendix
\section{Probability on Chains}\label{sec:ptc}

In this section, we address the mathematical derivation of the main concept introduced in section \ref{sec:ccd}, namely the chain-constrained domain. Let us first recall the definition of chains.

\begin{definition}
A binary relation ``$\le_R$'' on a set $X$ is a partial ordering of $X$ if it is transitive (i.e., $x\le_R y$ and $y\le_R z$ implies $x\le_R z$), reflexive (i.e., $x\le_R x$ for every $x\in X$), and anti-symmetric (i.e., $x\le_R y$ and $y\le_R x$ implies $x = y$). If $\le_R$ is a partial ordering on $X$, then $(X, \le_R)$ is called a partially ordered set.
\end{definition}

We do not necessarily need anti-symmetry; instead, if $x\le_R y$ and $y\le_R x$, we write $x =_R y$ and treat them as the equivalence class $\{y : y =_R x\}$. We use $x\ge_R y$ to denote $y\le_R x$, and we use $x <_R y$ to denote $x\le_R y$ and $x\ne_R y$.

\begin{definition}
A non-empty set $C\subseteq X$, where $(X, \le_R)$ is a partially ordered set, is a chain in $X$ if for any $x, y\in C$ we have either $x\le_R y$ or $y\le_R x$.
\end{definition}

Let $(X, \mathcal{A}, \mathbb{P})$ be a probability space. Suppose there exist $m$ binary relations $\le_k, k\in [m]$, such that for each $k$, $X$ is a chain (possibly without anti-symmetry) with respect to the partially ordered set $(X, \le_k)$. We assume the $\sigma$-algebra $\mathcal{A}$ contains sets of the form $\{y\in X : y =_k x\}$ and $\{y\in X : y\le_k x\}$ for each $x\in X$ and $k\in [m]$ (note that this is similar to the construction of the Borel $\sigma$-algebra of $\mathbb{R}$). More assumptions on measurability will be stated when needed.

Let $\mathbf{x} = (x_i)_i\in X^N$ be an $N$-tuple, where $N\ge m$. We define $h_k(\mathbf{x})$ to be the minimum of $x_i$'s under the relation $\le_k$, i.e., $h_k$ is an entry $y$ of $\mathbf{x}$ such that $y\le_k x_i$ for all $i\in [N]$. Consider the function
\[ \mathfrak{D}(\mathbf{x}) = \mathbb{P}\{x\in X : h_k\le_k x \text{ for all } k\in [m]\}; \]
in plain words, the function $\mathfrak{D}$ maps the tuple $\mathbf{x}$ to the measure of the set of $x$ that is no less than the minimum entry of $\mathbf{x}$ for each relation $\le_k$. For example, suppose $1 <_1 2 <_1 3$ and $3 <_2 2 <_2 1$ and $\mathbf{x} = (1, 3)$, then
\[ \{x : h_k\le_k x \text{ for } k = 1, 2\} = \{1, 2, 3\}. \]
Throughout this section, we are interested in the probability
\begin{equation} \label{ptc1}
\mathbb{P}^N\{\mathbf{x}\in X^N : \mathfrak{D}(\mathbf{x}) < 1 - \alpha\},
\end{equation}
where $\alpha\in (0, 1)$ and $\mathbb{P}^N$ is the completion of the (unique) product measure $\prod_{i = 1}^N \mathbb{P}$.

\begin{assumption} \label{as1}
The function $\mathfrak{D}$ is a random variable (with respect to the product space).
\end{assumption}

\subsection{The ``atomless'' case}

We first bound \eqref{ptc1} under the condition
\begin{equation} \label{ptc2}
\mathbb{P}^N\{\mathbf{x}\in X^N : x_i =_k x_{i'} \} = 0\stab \forall k\in [m], i \ne i'\in [N].
\end{equation}
The key step is to partition the event $\{1 - \mathfrak{D} > \alpha\}$. Observe that
\[ 1 - \mathfrak{D}(\mathbf{x}) = \mathbb{P}\{x\in X : x <_k h_k \text{ for some }k\in [m]\}. \]
For each $\mathbf{x}\in X^N$, let $\tilde{\mathbf{x}} = (\tilde{x}_i)_i$ be a permutation of $\mathbf{x}$ such that for each $k\in [m]$, $\tilde{x}_k$ is the minimum of $\tilde{x}_i, k\le i\le N$, under the relation $\le_k$ (i.e., $\tilde{x}_k\le_k \tilde{x}_{k'}\ \forall k'\ge k$). For example, suppose $3=_1 2 <_1 1$ and $1 <_2 2 <_2 3$ and $\mathbf{x} = (1, 2, 3)$, then the possible choices of $\tilde{\mathbf{x}}$ are $(3, 1, 2)$ and $(2, 1, 3)$. Certainly, under condition \eqref{ptc2}, there is a unique choice of the first $m$ indices of $\tilde{\mathbf{x}}$ almost surely. Now, we define $\delta_0 = 0$, and for each $k\in [m]$, define
\[ \delta_k(\mathbf{x}) = \mathbb{P}\{x\in X : x <_j \tilde{x}_j \text{ for some } j\in [k]\}. \]

\begin{assumption} \label{as2}
Each $\delta_k$ is a random variable.
\end{assumption}

Note that $\delta_k(\mathbf{x})$ is nondecreasing in $k$ for each $\mathbf{x}$. Consider the disjoint events
\[ E_k = \{\delta_{k - 1}\le \alpha, \delta_k > \alpha\}, \stab k\in [m]. \]
Under condition \eqref{ptc2}, the measure of $E_k$ does not depend on the particular choice of $\tilde{\mathbf{x}}$. We claim that $\{1 - \mathfrak{D} > \alpha\}\subseteq \cup_{k\in [m]} E_k$. Indeed, let $\mathbf{x}\in X^N$ such that $1 - \mathfrak{D}(\mathbf{x}) > \alpha$; since $h_k\le_k \tilde{x}_k$ for each $k\in [m]$, we have $\delta_m(\mathbf{x})\ge 1 - \mathfrak{D}(\mathbf{x}) > \alpha$; since $\delta_0 = 0$, there exists $k'$ such that $\delta_{k' - 1}(\mathbf{x})\le \alpha$ and $\delta_{k'}(\mathbf{x}) > \alpha$, i.e., $\mathbf{x}\in E_{k'}$. Hence,
\[ \mathbb{P}^N \{1 - \mathfrak{D} > \alpha\}\le \mathbb{P}^N(\cup_{k = 1}^m E_k) = \sum_{k = 1}^m \mathbb{P}^N(E_k). \]

It remains to bound the probability $\mathbb{P}^N(E_k)$. For each $k\in [m]$, we define
\[ S_k(x) = \mathbb{P}\{y\in X : y\ge_k x\}. \]
\begin{assumption} \label{as3}
The sets $\{x\in X : S_k(x) = \beta\}, \beta\in [0, 1],$ are measurable.
\end{assumption}
\begin{assumption} \label{as4}
If $Y\subseteq X$ is a measurable set with $\mathbb{P}(Y) > 0$, then\
\[ \sup_{x, y\in Y} \mathbb{P}\{z\in Y : x \le_k z \le_k y\} > 0. \]
\end{assumption}

Condition \eqref{ptc2} implies $\mathbb{P}\{y\in X : y =_1 x\} = 0$ for each $x\in X$; combining this fact with assumptions \ref{as3} and \ref{as4}, one can verify
\[ \mathbb{P}\{x\in X : S_1(x) = \beta\} = 0\stab \forall\beta\in [0, 1]. \]
Indeed, suppose for contradiction that $\mathbb{P}\{x\in X : S_1(x) = \beta\} > 0$; by assumption \ref{as4}, there exists $x <_1 y$ such that $S_1(x) = S_1(y) = \beta$ and $\mathbb{P}\{z\in X: x\le_1 z\le_1 y\} > 0$, but this also implies $S_1(x) = S_1(y) + \mathbb{P}\{z\in X: x\le_1 z <_1 y\} > \beta$, a contradiction.

Observe that
\[ \mathbb{P}^N(E_1) = \mathbb{P}^N\{1 - \delta_1 < 1 - \alpha\} = \left(\mathbb{P}\{x\in X: S_1(x) < 1 - \alpha\}\right)^N. \]
Let $\beta_0 = \sup \{S_1(x) : S_1(x) < 1 - \alpha, x\in X\}$, then $\beta_0\le 1 - \alpha$. If there exists $x_0$ such that $S_1(x_0) = \beta_0$, then
\[ \mathbb{P}\{x\in X: S_1(x) < 1 - \alpha\}\le \mathbb{P}\{x\in X : S_1(x) = \beta_0\} + S_1(x_0)\le 1 - \alpha; \]
otherwise, we can find a sequence $\{y_n\}$ such that $y_1 >_1 y_2 >_1 \ldots$, $S_1(y_n)\uparrow \beta_0$, and
\[ \mathbb{P}\{x\in X: S_1(x) < 1 - \alpha\} = \mathbb{P}\left(\cup_{n = 1}^\infty \{y\in X: y\ge_1 y_n\}\right) = \lim_n S_1(y_n)\le 1 - \alpha. \]
A similar argument gives $\mathbb{P}\{x\in X: S_1(x)\ge 1 - \alpha\}\le \alpha$. Thus,
\[ \mathbb{P}^N(E_1) = \left(\mathbb{P}\{x\in X: S_1(x) < 1 - \alpha\}\right)^N = (1 - \alpha)^N. \]
In fact, we can replace $\alpha$ by $\beta\in [0, 1]$ in the argument above to derive
\[ \mathbb{P}\{x\in X: S_1(x) < 1 - \beta\} = 1 - \beta\stab \forall\beta\in [0, 1]. \]

\begin{assumption} \label{as5}
The set $\{\mathbf{x}\in X^N : x_j =_j \tilde{x}_j, j\in [k]\}$ is measurable.
\end{assumption}

We next bound $\mathbb{P}^N(E_{k + 1})$ for $k\ge 1$. Consider the event
\begin{align*}
E_{k + 1}^* &= E_{k + 1}\cap \{\mathbf{x}\in X^N : x_j =_j \tilde{x}_j, j\in [k]\} \\
			&= \{\mathbf{x}\in X^N : \delta_k(\mathbf{x})\le \alpha, \delta_{k + 1}(\mathbf{x}) > \alpha, x_j =_j \tilde{x}_j, j\in [k]\}.
\end{align*}
By permutation,
\begin{equation} \label{ptc5}
\mathbb{P}^N(E_{k + 1}) = \frac{N!}{(N - k)!}\cdot \mathbb{P}^N(E_{k + 1}^*).
\end{equation}
For each $\mathbf{x}_k\in X^k$, consider the $\mathbf{x}_k$-section of $E_{k + 1}^*$:
\[ (E_{k + 1}^*)_{\mathbf{x}_k} = \{\mathbf{u}_{N - k}\in X^{N - k} : (\mathbf{x}_k, \mathbf{u}_{N - k})\in E_{k + 1}^*\}, \]
then the Tonelli's theorem implies
\begin{equation} \label{ptc4}
\mathbb{P}^N(E_{k + 1}^*) = \int \mathbb{P}^{N - k} (E_{k + 1}^*)_{\mathbf{x}_k}\ d\mathbb{P}^k(\mathbf{x}_k).
\end{equation}
If we can show that
\begin{enumerate}
\item\label{item1} the probability $\mathbb{P}^{N - k} (E_{k + 1}^*)_{\mathbf{x}_k}$ is either $0$ or $(1 - \alpha)^{N - k}$;
\item\label{item2} the set of $\mathbf{x}_k$ with $\mathbb{P}^{N - k} (E_{k + 1}^*)_{\mathbf{x}_k} = (1 - \alpha)^{N - k}$ has measure $\frac{\alpha^k}{k!}$,
\end{enumerate}
then \eqref{ptc5} and \eqref{ptc4} together yield
\[ \mathbb{P}^N(E_{k + 1}) = \binom{N}{k} \alpha^k (1 - \alpha)^{N - k}. \]

To show item \eqref{item1}, let us fix an $\mathbf{x}_k = (x_j)_j\in X^k$ such that $x_j\le_j x_{j'}$ for $j\le j'$, and
\[ \mathbb{P}\{x\in X : x <_j x_j \text{ for some } j\in [k]\} = \alpha_{\mathbf{x}_k}\le \alpha. \]
Consider
\begin{align*}
X_{\mathbf{x}_k} &= X\setminus \{x\in X : x <_j x_j \text{ for some } j\in [k]\} \\
								&= \{x\in X : x \ge_j x_j \text{ for all } j\in [k]\}
\end{align*}
and
\[ S_{\mathbf{x}_k}(x) = \mathbb{P}\{y\in X_{\mathbf{x}_k} : y\ge_{k + 1} x\}. \]
By condition \eqref{ptc2} and assumptions \ref{as3} and \ref{as4}, the set $\{x\in X_{\mathbf{x}_k} : S_{\mathbf{x}_k}(x) = \beta\}$ is measurable and has measure $0$ for each $\beta\in [0, 1 - \alpha_{\mathbf{x}_k}]$. Using an argument similar to that in the $E_1$ case, we can show that
\begin{equation} \label{ptc:lemeq}
\mathbb{P}\{x\in X_{\mathbf{x}_k}: S_{\mathbf{x}_k}(x) < 1 - \beta\} = 1 - \beta\stab \forall \beta\in [0, 1 - \alpha_{\mathbf{x}_k}].
\end{equation}
It follows that
\[ \mathbb{P}^{N - k} (E_{k + 1}^*)_{\mathbf{x}_k} = \left(\mathbb{P}\{x\in X_{\mathbf{x}_k}: S_{\mathbf{x}_k}(x) < (1 - \alpha_{\mathbf{x}_k}) - (\alpha - \alpha_{\mathbf{x}_k})\}\right)^{N - k} = (1 - \alpha)^{N - k}. \]
Note that $\mathbb{P}^{N - k} (E_{k + 1}^*)_{\mathbf{x}_k} = 0$ for other $\mathbf{x}_k$'s.

To show item \eqref{item2}, observe that
\begin{align*}
	&\mathbb{P}^k\{\mathbf{x}_k\in X^k : \mathbb{P}^{N - k} (E_{k + 1}^*)_{\mathbf{x}_k} = (1 - \alpha)^{N - k}\} \\
=\ &\mathbb{P}^k\{\mathbf{x}_k\in X^k : \mathbb{P}(X_{\mathbf{x}_k}^c)\le \alpha, x_j\le_j x_{j'} \text{ for all } j\le j'\} \\
=\ &\int_{x_1\in X : \mathbb{P}(X_{\mathbf{x}_1}^c)\le \alpha} \int_{x_2\in X_{\mathbf{x}_1} : \mathbb{P}(X_{\mathbf{x}_2}^c)\le \alpha}\ldots \int_{x_k\in X_{\mathbf{x}_{k - 1}} : \mathbb{P}(X_{\mathbf{x}_k}^c)\le \alpha} \mathbf{1}\ d\mathbb{P}(x_k)\ldots d\mathbb{P}(x_1),
\end{align*}
where the second equality follows by applying the Tonelli's theorem $k$ times. In the multiple integral, $\mathbf{x}_j\in X^j$ is the $j$-tuple with entries $x_1, \ldots, x_j$; in particular, if $j\le j'$, then $\mathbf{x}_j$ is the first $j$ entries of $\mathbf{x}_{j'}$. For each $j\in [k]$, we denote $t_j$ to be $\mathbb{P}(X_{\mathbf{x}_j}^c) - \mathbb{P}(X_{\mathbf{x}_{j - 1}}^c)$. Let us fix an appropriate $\mathbf{x}_{j - 1}$ (that appears in the multiple integral), then $\sum_{i = 1}^{j - 1} t_i\le \alpha$ is fixed. Since
\[ t_j = \mathbb{P}(X_{\mathbf{x}_j}^c\setminus X_{\mathbf{x}_{j - 1}}^c) = \mathbb{P}\{x\in X_{\mathbf{x}_{j - 1}} : x <_j x_j\} = 1 - S_{\mathbf{x}_{j - 1}}(x_j), \]
by \eqref{ptc:lemeq}, for $\beta\in [0, 1 - \sum_{i = 1}^{j - 1} t_i]$ we have
\[ \mathbb{P}\left\{x_j\in X_{\mathbf{x}_{j - 1}} : t_j\le \beta\right\} = \mathbb{P}\left\{x_j\in X_{\mathbf{x}_{j - 1}} : S_{\mathbf{x}_{j - 1}}(x_j)\ge 1 - \beta\right\} = \beta; \]
in particular, let $d_j = \alpha - \sum_{i = 1}^{j - 1} t_i$ and $n\in \mathbb{N}$, then
\[ \mathbb{P}\left\{x_j\in X_{\mathbf{x}_{j - 1}} : \frac{(a - 1)d_j}{n} < t_j\le \frac{a d_j}{n}\right\} = \frac{d_j}{n}\stab \forall a\in [n]. \]
If we calculate the integral by simple functions on the sets
\[ \left\{x_j\in X_{\mathbf{x}_{j - 1}} : \frac{(a - 1)d_j}{n} < t_j\le \frac{a d_j}{n}\right\},\stab j\in [k], n\in \mathbb{N}, a\in [n], \]
then
\begin{align*}
  &\int_{x_1\in X : \mathbb{P}(X_{\mathbf{x}_1}^c)\le \alpha} \int_{x_2\in X_{\mathbf{x}_1} : \mathbb{P}(X_{\mathbf{x}_2}^c)\le \alpha}\ldots \int_{x_k\in X_{\mathbf{x}_{k - 1}} : \mathbb{P}(X_{\mathbf{x}_k}^c)\le \alpha} \mathbf{1}\ d\mathbb{P}(x_k)\ldots d\mathbb{P}(x_1) \\
=\ &\int_{x_1\in X : t_1\le \alpha} \int_{x_2\in X_{\mathbf{x}_1} : \sum_{j = 1}^2 t_j\le \alpha}\ldots \int_0^{\alpha - \sum_{j = 1}^{k - 1} t_j} 1\ dt_k\ldots d\mathbb{P}(x_2) d\mathbb{P}(x_1) \\
\vdots\ & \\
=\ &\int_0^\alpha \int_0^{\alpha - t_1}\ldots \int_0^{\alpha - \sum_{j = 1}^{k - 1} t_j} 1\ dt_k\ldots dt_2 dt_1 \\
=\ &\frac{\alpha^k}{k!}.
\end{align*}
Thus far, we have proved the following lemma:

\begin{lemma} \label{ptclem}
Let $N\ge m$ and $\alpha\in [0, 1]$. Suppose 
\[ \mathbb{P}^N\{\mathbf{x}\in X^N : x_i =_k x_{i'} \} = 0\stab \forall k\in [m], i \ne i'\in [N], \]
and assumptions \ref{as1} to \ref{as5} are satisfied, then
\[ \mathbb{P}^N\{\mathbf{x}\in X^N : \mathfrak{D}(\mathbf{x}) < 1 - \alpha\}\le \sum_{k = 0}^{m - 1} \binom{N}{k} \alpha^k (1 - \alpha)^{N - k}. \]
\end{lemma}

\begin{remark}
{\rm
Assumption \ref{as4} is the only assumption that is not related to measurability. Without this assumption, it may happen that
\[ \mathbb{P}\{y\in X : y <_R x\} = 1 \]
for every $x\in X$. The following construction is based on the example given in \cite{R19}, and it relies on the axiom of choice: by the well ordering principle, there exists the set of countable ordinals $\Omega$ that is well ordered by $\le_w$ \cite{F99} (note that $\Omega$ itself is uncountable); consider the probability space $(\Omega, \mathcal{M}, \mu)$, where
\[ \mathcal{M} = \{U\subseteq \Omega : U \text{ countable or } U^c \text{ countable}\}, \]
and $\mu(U) = 0$ if $U$ is countable and $\mu(U) = 1$ if $U^c$ is countable. Let $\le_R$ be the converse of $\le_w$ (i.e., $x\le_R y\iff y\le_w x$); it can be verified that $\mathcal{M}$ contains sets of the form $\{y =_R x\}$ and $\{y\le_R x\}$ for each $x\in \Omega$. Since the set
\[ \{y\in \Omega : y\ge_R x\} = \{y\in \Omega : y\le_w x\} \]
is countable for every $x\in \Omega$, we have
\[ \mu\{y\in \Omega: y <_R x\} = 1 - \mu\{y\in \Omega: y\ge_R x\} = 1. \]
}
\end{remark}

\subsection{The general case}
We deal with the general case by constructing a new probability space that satisfies condition \eqref{ptc2}. Let $([0, 1], \mathcal{L}, \lambda)$ be the probability space such that $\mathcal{L}$ is the Lebesgue $\sigma$-algebra on $[0, 1]$, and $\lambda$ is the Lebesgue measure. Consider the product space $(\mathfrak{X}, \mathcal{A}\otimes\mathcal{L}, \mathbb{P}_\lambda)$, where $\mathfrak{X} = X\times [0, 1]$, and $\mathbb{P}_\lambda = \mathbb{P}\times \lambda$ is the product measure. For each $k\in [m]$, we define a binary relation $\le_k$ on $\mathfrak{X}$: we say $(x, s)\le_k (y, t)$ if either $x <_k y$, or $x =_k y$ and $s\le t$; note that $\mathfrak{X}$ is a chain (possibly without anti-symmetry) with respect to the partially ordered set $(\mathfrak{X}, \le_k)$. It can be verified that the $\sigma$-algebra $\mathcal{A}\otimes \mathcal{L}$ contains sets of the form $\{(y, t) : (y, t) =_k (x, s)\}$ and $\{(y, t) : (y, t)\le_k (x, s)\}$ for each $(x, s)\in \mathfrak{X}$ and $k\in [m]$.

We denote $\mathbb{P}_\lambda^N$ to be the completion of the product measure $\prod_{i = 1}^N \mathbb{P}_\lambda$. Observe that
\[ \{(\mathbf{x}, \mathbf{s})\in \mathfrak{X}^N : (x_i, s_i) =_k (x_{i'}, s_{i'})\}\subseteq \{(\mathbf{x}, \mathbf{s})\in \mathfrak{X}^N : s_i = s_{i'}\}; \]
since the latter is a $\mathbb{P}_\lambda^N$-null set when $i\ne i'$, we have
\[ \mathbb{P}_\lambda^N\{(\mathbf{x}, \mathbf{s})\in \mathfrak{X}^N : (x_i, s_i) =_k (x_{i'}, s_{i'})\} = 0\stab \forall k\in [m], i\ne i'\in [N]. \]
We define the function $\mathfrak{D}_\lambda$ on $\mathfrak{X}^N$ similarly, i.e., let $(\mathbf{x}, \mathbf{s})\in \mathfrak{X}^N$ and denote $h_k(\mathbf{x}, \mathbf{s})$ to be the minimum of $(x_i, s_i)'s$ under the relation $\le_k$, then
\[ \mathfrak{D}_\lambda(\mathbf{x}, \mathbf{s}) = \mathbb{P}_\lambda\{(x, s)\in \mathfrak{X} : h_k\le_k (x, s) \text{ for all } k\in [m]\}. \]
If $(\mathfrak{X}, \mathcal{A}\otimes\mathcal{L}, \mathbb{P}_\lambda)$ satisfies assumptions \ref{as2} to \ref{as5}, then by the proof of Lemma \ref{ptclem},
\[ \mathbb{P}_\lambda^N(\cup_{k = 1}^m E_k) = \sum_{k = 1}^m \mathbb{P}_\lambda^N(E_k) = \sum_{k = 0}^{m - 1} \binom{N}{k} \alpha^k (1 - \alpha)^{N - k}, \]
where $E_k$'s are defined with respect to $\mathfrak{X}^N$. Since
\[ \mathfrak{D}_\lambda(\mathbf{x}, \mathbf{s})\le \mathfrak{D}_\lambda(\mathbf{x}, \mathbf{0}) = \mathfrak{D}(\mathbf{x}), \]
we have
\[ (\{1 - \mathfrak{D} > \alpha\}\times [0, 1]^N)\subseteq \{1 - \mathfrak{D}_\lambda > \alpha\}\subseteq \cup_{k = 1}^m E_k. \]
Therefore,
\[ \mathbb{P}^N\{\mathfrak{D} < 1 - \alpha\}\le \mathbb{P}_\lambda^N(\cup_{k = 1}^m E_k) = \sum_{k = 0}^{m - 1} \binom{N}{k} \alpha^k (1 - \alpha)^{N - k}. \]

\begin{theorem} \label{ptcthm}
Let $N\ge m$ and $\alpha\in [0, 1]$. Suppose $\mathfrak{D}$ satisfies assumption \ref{as1}, and $(\mathfrak{X}, \mathcal{A}\otimes\mathcal{L}, \mathbb{P}_\lambda)$ satisfies assumptions \ref{as2} to \ref{as5}, then
\begin{equation} \label{ptc:prob}
\mathbb{P}^N\{\mathbf{x}\in X^N : \mathfrak{D}(\mathbf{x}) < 1 - \alpha\} \le \sum_{k = 0}^{m - 1} \binom{N}{k} \alpha^k (1 - \alpha)^{N - k}.
\end{equation}
\end{theorem}

\begin{example}
{\rm
One natural example of chains on probability space is the one where the orders are preserved by the inverse of random variables. Suppose there are $m$ random variables $\tau_k, k\in [m]$, on the probability space $(X, \mathcal{A}, \mathbb{P})$, such that for each $k\in [m]$,
\[ \tau_k(x)\le \tau_k(y) \implies x\le_k y, \stab x, y\in X; \]
certainly, it does not preclude the case where $\tau_k(x) < \tau_k(y)$ but $x =_k y$. For each $x\in X$, the equivalence class $\ip{x}_k = \{y\in X : y =_k x\}$ is the preimage of some intervals (open, closed, or half-open) under $\tau_k$; indeed, let $y =_k x$ and $\tau_k(y)\le \tau_k(x)$, then $z =_k x$ if $\tau_k(z)\in [\tau_k(y), \tau_k(x)]$, thus $\ip{x}_k$ is the preimage of $\text{conv }\tau_k(\ip{x}_k)$, i.e., the convex hull of $\tau_k(\ip{x}_k)$; this also implies there are only countably many $\ip{x}_k$ such that $\tau_k(\ip{x}_k)$ is not a singleton. Similarly, the set $\{y\in X : y \le_k x\}$ is the preimage of some half-infinite intervals under $\tau_k$. For $s\in \mathbb{R}$, we denote $\ip{s}_k$ to be $\{x\in X : s\in \text{conv }\tau_k(\ip{x}_k)\}$; when $\ip{s}_k$ is non-empty, it coincides with some $\ip{x}_k$. We first show that assumption \ref{as1} is satisfied, i.e., $\mathfrak{D}$ is a random variable.

Assumption \ref{as1}: We need to show $\{\mathfrak{D}\le \alpha\}$ is measurable for every $\alpha\in [0, 1)$. For each $k\in [m]$ and $p\in \mathbb{Q}$, we define
\[ B_{k, p} = \begin{cases} \{y\in X : y \ge_k x\} & \text{if } p\in \text{conv }\tau_k(\ip{x}_k) \\ \tau_k^{-1}(p, \infty) & \text{if no such } x \text{ exists} \end{cases}, \]
and consider the collection of sets
\begin{align*}
\mathfrak{B}_k &= \{B_{k, p} : p\in \mathbb{Q}\}\cup \{\{y\in X : y\ge_k x\} : \mathbb{P}(\ip{x}_k) > 0, x\in X\} \\
		&\tab \cup \{\{y\in X : y >_k x\} : \mathbb{P}(\ip{x}_k) > 0, x\in X\};
\end{align*}
each $\mathfrak{B}_k$ is countable since there can be at most countably many $\ip{x}_k$ such that $\mathbb{P}(\ip{x}_k) > 0$. Let
\[ \mathfrak{E}_n = \bigcup_{B_k\in \mathfrak{B}_k : \mathbb{P}(\cap_{k = 1}^m B_k)\le \alpha + \frac{1}{n}} \left(\cap_{k = 1}^m B_k\right)^N. \]
We claim that $\{\mathfrak{D}\le \alpha\} = \cap_{n = 1}^\infty \mathfrak{E}_n$. To see $\{\mathfrak{D}\le \alpha\}\supseteq \cap_{n = 1}^\infty \mathfrak{E}_n$, let $\mathbf{x}\in \mathfrak{E}_n$, then
\[ \mathfrak{D}(\mathbf{x}) = \mathbb{P}(\cap_{k = 1}^m \{x\in X : h_k\le_k x\})\le \mathbb{P}(\cap_{k = 1}^m B_k)\le \alpha + \frac{1}{n}, \]
thus $\mathfrak{D}(\mathbf{x})\le \alpha$ for each $\mathbf{x}\in \cap_{n = 1}^\infty \mathfrak{E}_n$. On the other hand, let $\mathbf{x}\in X^N, \mathfrak{D}(\mathbf{x})\le \alpha$; for each $k\in [m]$, if $\mathbb{P}(\ip{h_k}_k) > 0$ or there exists rational $p_k\in \text{conv } \tau_k(\ip{h_k}_k)$, then we set $B_k = \{y\in X : y\ge_k h_k\}$; otherwise, we can find a rational sequence $p_i\uparrow \tau_k(h_k)$, and $\mathbb{P}(\tau_k^{-1}[p_i, \tau_k(h_k)])\downarrow \mathbb{P}(\tau_k^{-1}\{\tau_k(h_k)\}) = 0$, hence there exists rational $p_0 < \tau_k(h_k)$ such that $\mathbb{P}(\tau_k^{-1}[p_0, \tau_k(h_k)])\le \frac{1}{mn}$; if there exists $y_k\in X$ such that $\mathbb{P}(\ip{y_k}_k) > 0$ and $p_0\in \text{conv }\tau_k(\ip{y_k}_k)$, then we set $B_k = \{y\in X : y >_k y_k\}$, else we set $B_k = B_{k, p_0}$; either way, we have $\mathbf{x}\in (\cap_{k = 1}^m B_k)^N$, and
\begin{align*}
\mathfrak{D}(\mathbf{x}) &= \mathbb{P}(\cap_{k = 1}^m \{y\in X : y\ge_k h_k\}) \\
		&= \mathbb{P}(\cap_{k = 1}^m B_k\setminus (B_k\setminus \{y\in X : y\ge_k h_k\})) \\
		&\ge \mathbb{P}(\cap_{k = 1}^m B_k) - \sum_{k = 1}^m \frac{1}{mn}.
\end{align*}
Hence, $\mathbf{x}\in \mathfrak{E}_n$ for each $n\in \mathbb{N}$. Since each $\mathfrak{E}_n$ is measurable, $\{\mathfrak{D}\le \alpha\}$ is measurable as desired.

Now, consider the product space $(\mathfrak{X}, \mathcal{A}\otimes \mathcal{L}, \mathbb{P}_\lambda)$. Let $\le_\ell$ be the lexicographical order on $\mathbb{R}\times [0, 1]$; note that $(\tau_k(x), s)\le_\ell (\tau_k(y), t)$ does {\it not} imply $(x, s)\le_k (y, t)$ in general, though it is still valid that
\[ (\tau_k(x), s)\le_\ell (\tau_k(y), t)\implies x\le_k y, \stab (x, s), (y, t)\in \mathfrak{X}. \]
We show below that assumptions \ref{as2} to \ref{as5} are satisfied automatically.

Assumption \ref{as3}: Suppose the set $\{(x, s)\in \mathfrak{X} : S_k(x, s) = \beta\}$ is non-empty, let
\[ (v_1, v_2) = \inf_{S_k(x, s) = \beta} (\tau_k(x), s)\tab (w_1, w_2) = \sup_{S_k(x, s) = \beta} (\tau_k(x), s); \]
here, the sup and the inf are taken with respect to the lexcographical order. In the case where the sup and the inf are attained,
\begin{align*}
\{(x, s) : S_k(x, s) = \beta\} &= (\tau_k^{-1}[v_1, w_1]\times [0, 1]) \setminus (\left(\ip{v_1}_k\times [0, v_2)\right) \cup \left(\ip{w_1}_k\times (w_2, 1]\right)),
\end{align*}
thus the set is measurable. The other cases are similar.

Assumption \ref{as4}: Suppose $Y\subseteq \mathfrak{X}$ is a measurable set with $\mathbb{P}_\lambda(Y) > 0$. Consider
\[ \hat{Y}_k = \{(\tau_k(x), s) : (x, s)\in Y\}\subseteq \mathbb{R}\times [0, 1]; \]
then there exist two sequences $\hat{\mathbf{v}}^1\ge_\ell \hat{\mathbf{v}}^2\ge_\ell\ldots$ and $\hat{\mathbf{w}}^1\le_\ell \hat{\mathbf{w}}^2\le_\ell\ldots$ in $\hat{Y}_k$, such that
\[ \hat{Y}_k = \cup_{n = 1}^\infty \{\hat{\mathbf{y}}\in \hat{Y}_k : \hat{\mathbf{v}}^n\le_\ell \hat{\mathbf{y}}\le_\ell \hat{\mathbf{w}}^n\}. \]
For each $n$, we pick an $\mathbf{v}^n\in Y$ such that $\mathbf{v}^n\le_k \mathbf{v}^{n - 1}$, $v^n_1 =_k \tau_k^{-1}(\hat{v}^n_1)$, and
\[ v^n_2\le \inf_{(x, s)\in Y : x =_k v^n_1} s + \frac{1}{n}; \]
similarly, we pick an $\mathbf{w}^n\in Y$ such that $\mathbf{w}^n\ge_k \mathbf{w}^{n - 1}$, $w^n_1 =_k \tau_k^{-1}(\hat{w}^n_1)$, and
\[ w^n_2\ge \sup_{(x, s)\in Y : x =_k w^n_1} s - \frac{1}{n}; \]
then $\mathbf{v}^1\ge_k \mathbf{v}^2\ge_k\ldots$ and $\mathbf{w}^1\le_k \mathbf{w}^2\le_k\ldots$, and
\[ Y_o = Y\setminus \left(\cup_{n = 1}^\infty \{(x, s)\in Y : \mathbf{v}^n\le_k (x, s) \le_k \mathbf{w}^n\}\right) \]
is a $\mathbb{P}_\lambda$-null set; indeed, let $(x, s)\in Y$, then $\hat{\mathbf{v}}^{n_0}\le_\ell (\tau_k(x), s)\le_\ell \hat{\mathbf{w}}^{n_0}$ for some $n_0\in \mathbb{N}$, which implies $v^n_1\le_k x\le_k w^n_1$ for any $n\ge n_0$; note that $(x, s)$ can only be in $Y_o$ when either $x =_k w_1^n$ for all $n\ge n_0$ and $w^n_2\uparrow s$, or $x =_k v_1^n$ for all $n\ge n_0$ and $v^n_2\downarrow s$, thus $Y_o\subseteq X\times \{s_1, s_2\}$ for some $s_1, s_2\in [0, 1]$, which has measure $0$. Therefore, by continuity from below,
\[ \sup_{\mathbf{v}, \mathbf{w}\in Y} \mathbb{P}_\lambda\{(x, s)\in Y : \mathbf{v}\le_k (x, s) \le_k \mathbf{w}\} = \mathbb{P}_\lambda(Y) > 0. \]

Assumption \ref{as5}: Since
\[ \mathfrak{N}_k = \{(\mathbf{x}, \mathbf{s})\in \mathfrak{X}^N : (x_i, s_i) =_j (x_{i'}, s_{i'}) \text{ for some } j\in [k], i\ne i'\in [N]\} \]
is a $\mathbb{P}_\lambda^N$-null set, we only need to show that
\[ \mathfrak{S}_k = \{(\mathbf{x}, \mathbf{s})\in \mathfrak{X}^N : (x_j, s_j) <_j (x_{j'}, s_{j'}) \text{ for all } j < j', j\in [k], j'\in [m]\} \]
is measurable. It suffices to show $\{(\mathbf{x}, \mathbf{s})\in \mathfrak{X}^N : (x_j, s_j) <_j (x_{j'}, s_{j'})\}$ is measurable for each $j < j', j\in [k], j'\in [m]$. By definition, $(x_j, s_j) <_j (x_{j'}, s_{j'})$ if and only if either $x_j <_j x_{j'}$, or $x_j =_j x_{j'}$ and $s_j < s_{j'}$. Since $\{x_j <_j x_{j'}\}$ is the permutation of the set
\[ \cup_{q\in \mathbb{Q}} \left(\tau_j^{-1}(-\infty, q]\times \tau_j^{-1}(q, \infty)\setminus \ip{q}_j\times \ip{q}_j\right)\times X^{N - 2} \times [0, 1]^N, \]
$\{x_j <_j x_{j'}\}$ is measurable. Also, $\{x_j =_j x_{j'}\} = (\{x_j <_j x_{j'}\}\cup \{x_j >_j x_{j'}\})^c$ is measurable, hence
\[ \{(x_j, s_j) <_j (x_{j'}, s_{j'})\} = \{x_j <_j x_{j'}\}\cup (\{x_j =_j x_{j'}\}\cap \{s_j < s_{j'}\}) \]
is measurable.

Assumption \ref{as2}: We need to show that $\{\delta_k\ge \alpha\}$ is measurable for every $\alpha\in (0, 1]$ and $k\in [m]$. If we can show that $\{\delta_k\ge \alpha\}\cap \mathfrak{S}_k$ is measurable, then $\{\delta_k\ge \alpha\}$ is measurable by permutation. For each $j\in [k]$ and $p, q\in \mathbb{Q}$, we define
\[ B_{j, p, q} = \begin{cases} \{(y, t)\in \mathfrak{X} : (y, t)\ge_j (x, q)\} & \text{if } p\in \text{conv}\ \tau_j(\ip{x}_j) \\ \tau_j^{-1}(p, \infty)\times [0, 1] & \text{if no such } x \text{ exists} \end{cases},\]
and consider the collection of sets
\[ \mathfrak{B}_j = \{B_{j, p, q} : p, q\in \mathbb{Q}\}\cup \{\{(y, t)\in \mathfrak{X} : (y, t)\ge_j (x, q)\} : \mathbb{P}(\ip{x}_j) > 0, x\in X, q\in \mathbb{Q}\}; \]
each $\mathfrak{B}_j$ is countable since there can be at most countably many $\ip{x}_j$ such that $\mathbb{P}(\ip{x}_j) > 0$. Let
\[ \mathfrak{A}_n = \left[\bigcup_{B_j\in \mathfrak{B}_j : \mathbb{P}_\lambda(\cup_{j = 1}^k B_j^c)\ge \alpha - \tfrac{1}{n}} \left(\prod_{j = 1}^k B_j\right) \times \mathfrak{X}^{N - k}\right]\cap \mathfrak{S}_k. \]
We claim that $\{\delta_k\ge \alpha\}\cap \mathfrak{S}_k = \cap_{n = 1}^\infty \mathfrak{A}_n$. To see that $\{\delta_k\ge \alpha\}\cap \mathfrak{S}_k \supseteq \cap_{n = 1}^\infty \mathfrak{A}_n$, let $(\mathbf{x}, \mathbf{s})\in \mathfrak{A}_n$, since $(\mathbf{x}, \mathbf{s})\in \mathfrak{S}_k$, we have
\[ \delta_k(\mathbf{x}, \mathbf{s}) = \mathbb{P}_\lambda(\cup_{j = 1}^k \{(y, t)\in \mathfrak{X} : (y, t) <_j (x_j, s_j)\})\ge \mathbb{P}_\lambda(\cup_{j = 1}^k B_j^c)\ge \alpha - \tfrac{1}{n}; \]
thus $\delta_k(\mathbf{x}, \mathbf{s})\ge \alpha$ for each $(\mathbf{x}, \mathbf{s})\in \cap_{n = 1}^\infty \mathfrak{A}_n$. On the other hand, let $(\mathbf{x}, \mathbf{s})\in \mathfrak{S}_k$ such that $\delta_k(\mathbf{x}, \mathbf{s})\ge \alpha$; for each $j\in [k]$, if $\mathbb{P}(\ip{x_j}_j) > 0$ or there exists rational $p_j\in \text{conv }\tau_j(\ip{x_j}_j)$, then we pick $q_j\in \mathbb{Q}, s_j\ge q_j\ge s_j - \frac{1}{kn}$, and set $B_j = \{(y, t)\ge_j (x_j, q_j)\}$; otherwise, we can find a rational sequence $p_i\uparrow \tau_j(x_j)$, such that $\mathbb{P}(\tau_j^{-1}[p_i, \tau_j(x_j)]) \downarrow \mathbb{P}(\tau_j^{-1}\{\tau_j(x_j)\}) = 0$, hence there exists rational $p_0 < \tau_j(x_j)$ such that $\mathbb{P}(\tau_j^{-1}[p_0, \tau_j(x_j)])\le \tfrac{1}{kn}$, and we set $B_j = B_{j, p_0, 1}$; either way, we have $(\mathbf{x}, \mathbf{s})\in \big(\prod_{j = 1}^k B_j\big)\times \mathfrak{X}^{N - k}$, and
\begin{align*}
\delta_k(\mathbf{x}, \mathbf{s}) &= \mathbb{P}_\lambda(\cup_{j = 1}^k \{(y, t)\in \mathfrak{X} : (y, t) <_j (x_j, s_j)\}) \\
		&= \mathbb{P}_\lambda(\cup_{j = 1}^k B_j^c\cup (B_j\setminus \{(y, t) : (y, t)\ge_j (x_j, s_j)\})) \\
		&\le \mathbb{P}_\lambda(\cup_{j = 1}^k B_j^c) + \sum_{j = 1}^k \frac{1}{kn}.
\end{align*}
Hence, $(\mathbf{x}, \mathbf{s})\in \mathfrak{A}_n$ for each $n\in \mathbb{N}$. Now, since each $\mathfrak{A}_n$ is measurable, $\{\delta_k\ge \alpha\}\cap \mathfrak{S}_k$ is measurable as desired.
}
\end{example}

\begin{corollary} \label{ptccor1}
Let $N\ge m$ and $\alpha\in [0, 1]$. Suppose there exist random variables $\tau_k, k\in [m]$, such that
\[ \tau_k(x)\le \tau_k(y)\implies x \le_k y,\stab x, y\in X, \]
then
\[ \mathbb{P}^N\{\mathbf{x}\in X^N : \mathfrak{D}(\mathbf{x}) < 1 - \alpha\} \le \sum_{k = 0}^{m - 1} \binom{N}{k} \alpha^k (1 - \alpha)^{N - k}. \]
\end{corollary}

We next give examples where the bound \eqref{ptc:prob} is tight.
\begin{example}
{\rm
Consider the probability space $([0, 1], \mathcal{L}, \lambda)$. Let $m\in \mathbb{N}$, we define $m$ binary relations $\le_k$ on $[0, 1]$: for each $k\in [m]$, we say $s <_k t$ if
\begin{itemize}
\item $s < t$ and $s, t\in \left[\tfrac{k - 1}{m}, \tfrac{k}{m}\right)$;
\item $s < t$ and $s, t\in [0, 1]\setminus \left[\tfrac{k - 1}{m}, \tfrac{k}{m}\right)$;
\item $s\in \left[\tfrac{k - 1}{m}, \tfrac{k}{m}\right)$ and $t\in [0, 1]\setminus \left[\tfrac{k - 1}{m}, \tfrac{k}{m}\right)$;
\end{itemize}
in other words, $\le_k$ is derived from the standard ordering $\le$ by making the interval $\left[\tfrac{k - 1}{m}, \tfrac{k}{m}\right)$ smaller. Let $N\ge m$ and $\alpha\in [0, \frac{1}{m}]$, we claim that
\[ \mathbb{\lambda}^N\{\mathbf{x}\in [0, 1]^N : \mathfrak{D}(\mathbf{x}) < 1 - \alpha\} = \sum_{k = 0}^{m - 1} \binom{N}{k} \alpha^k (1 - \alpha)^{N - k}. \]
The idea is similar to the proof of Lemma \ref{ptclem}. We define $\Delta_0 = 0$, and for each $k\in [m]$, define
\[ \Delta_k(\mathbf{x}) = \lambda\{x\in [0, 1] : x <_j h_j(\mathbf{x}) \text{ for some } j\in [k]\}, \]
where $h_j(\mathbf{x})$ is the minimum of $x_i$'s under the relation $\le_j$. Consider the disjoint events
\[ E_k = \{\Delta_{k - 1}\le \alpha, \Delta_k > \alpha\},\stab k\in [m], \]
then $\{\mathfrak{D} < 1 - \alpha\} = \cup_{k = 1}^m E_k$. For almost every $\mathbf{x}$ such that $\Delta_{k - 1}(\mathbf{x})\le \alpha\le \tfrac{1}{m}$, we have $h_j(\mathbf{x})\in \left[\tfrac{j - 1}{m}, \tfrac{j}{m}\right)$ for all $j\in [k - 1]$. Hence, $E_k$ has the same measure as the set
\begin{equation} \label{ptctight}
\left\{\mathbf{x}\in [0, 1]^N :
\begin{array}{l} h_j\in \left[\tfrac{j - 1}{m}, \tfrac{j}{m}\right)\ \forall j\in [k - 1], \sum_{j = 1}^{k - 1} \left(h_j - \tfrac{j - 1}{m}\right)\le \alpha, \\
h_k\not\in \left[\tfrac{k - 1}{m}, \tfrac{k - 1}{m} + \alpha - \sum_{j = 1}^{k - 1} \left(h_j - \tfrac{j - 1}{m}\right)\right)\end{array}\right\};
\end{equation}
the measure of \eqref{ptctight} is $\binom{N}{k - 1} \alpha^{k - 1}(1 - \alpha)^{N - k + 1}$, which can be calculated by considering the permutation $x_j = h_j, j\in [k - 1]$, and conditioning on the values of $x_j$.
}
\end{example}

So far, we have assumed $N\ge m$. As illustrated by the following example, there may be no formula that bounds the probability \eqref{ptc:prob} uniformly when $N < m$. Let $N = 1$ and $m = 2$, and consider again $([0, 1], \mathcal{L}, \lambda)$; the relation $\le_1$ is $\le$, and $\le_2$ is the converse of $\le$ (i.e., $s\le_2 t \iff t \le s$). For any $\alpha\in [0, 1)$, we have
\[ \mathbb{\lambda}\{\mathbf{x}\in X : \mathfrak{D}(\mathbf{x}) < 1 - \alpha\} = 1. \]

\subsection{Applications} \label{sec:ptc3}

A collection of sets $\left\{U^\omega\right\}_{\omega\in I}$ is a chain if for any $\omega_1, \omega_2\in I$, either $U^{\omega_1}\subseteq U^{\omega_2}$ or $U^{\omega_1}\supseteq U^{\omega_2}$. Let $\xi$ be a random vector supported on the set $\Xi$. Suppose we are given $m$ chains $\big\{U_k^\xi\big\}_{\xi\in \Xi}, k\in [m]$, then each chain $\big\{U_k^\xi\big\}_{\xi\in \Xi}$ induces a binary relation $\le_k$ on $\Xi$, i.e.,
\[ \omega_1\le_k \omega_2 \iff U_k^{\omega_1}\subseteq U_k^{\omega_2}\stab \omega_1, \omega_2\in \Xi. \]
Note that $\Xi$ is a chain with respect to the partially ordered set $(\Xi, \le_k)$ for each $k\in [m]$. Let $\xi^{[N]} = (\xi^1, \ldots, \xi^N)$ be the i.i.d. (random) sample of $\xi$ of size $N$, then
\[ \mathfrak{D}(\xi^{[N]}) = \mathbb{P}\left\{\xi\in \Xi : \cap_{i = 1}^N U_k^{\xi^i}\subseteq U_k^\xi \text{ for all } k\in [m]\right\}. \]
\begin{corollary} \label{ptccor2}
Under the assumptions of Theorem \ref{ptcthm},
\[ \mathbb{P}^N\left\{\mathfrak{D}(\xi^{[N]}) < 1 - \alpha\right\} \le \sum_{k = 0}^{m - 1} \binom{N}{k} \alpha^k (1 - \alpha)^{N - k}. \]
\end{corollary}
Oftentimes, $\big\{U_k^\xi\big\}$ are sublevel sets, i.e., $U_k^\xi = c_k^{-1}(-\infty, \ell_k(\xi)]$ for some function $c_k$ and random variable $\ell_k$ supported on $\Xi$. In this case,
\[ \ell_k(\omega_1)\le \ell_k(\omega_2)\implies U_k^{\omega_1}\subseteq U_k^{\omega_2}\implies \omega_1\le_k \omega_2, \stab \omega_1, \omega_2\in \Xi. \]
We can also define the relation
\[ \omega_1\le_k \omega_2\iff \ell_k(\omega_1)\le \ell_k(\omega_2), \stab \omega_1, \omega_2\in \Xi, \]
and
\[ \mathfrak{D}_r(\xi^{[N]}) = \mathbb{P}\left\{\xi\in \Xi : \min_{i\in [N]} \ell_k(\xi^i)\le \ell_k(\xi) \text{ for all } k\in [m]\right\}. \]
Note that
\[ \min_{i\in [N]} \ell_k(\xi^i)\le \ell_k(\xi) \implies \cap_{i = 1}^N U_k^{\xi^i}\subseteq U_k^\xi, \]
hence, $\mathfrak{D}(\xi^{[N]})\ge \mathfrak{D}_r(\xi^{[N]})$. Since the orders are preserved by the inverse of the random variables $\ell_k$, we can apply Corollary \ref{ptccor1} instead of Theorem \ref{ptcthm}.
\begin{corollary} \label{ptccor3}
If $U_k^\xi$'s are sublevel sets, then
\[ \mathbb{P}^N\left\{\mathfrak{D}(\xi^{[N]}) < 1 - \alpha\right\}\le \mathbb{P}^N\left\{\mathfrak{D}_r(\xi^{[N]}) < 1 - \alpha\right\} \le \sum_{k = 0}^{m - 1} \binom{N}{k} \alpha^k (1 - \alpha)^{N - k}. \]
\end{corollary}

\section*{Acknowledgments}
The author is indebted to Alexander Shapiro and the anonymous referees for constructive comments which helped to improve the presentation of the manuscript.

\nocite{*}
\bibliographystyle{amsplain}
\bibliography{references}

\end{document}